\documentclass[10pt,letterpaper]{amsart}
\usepackage[utf8]{inputenc}
\usepackage[english]{babel}
\usepackage{subfigure}
\usepackage{amsmath}
\usepackage{amsthm} 
\usepackage{hyperref} 
\usepackage{amsfonts}
\usepackage{tikz-cd}
\usepackage{amssymb}
\usepackage{csquotes}
\usepackage[
backend=biber,
style=alphabetic,
sorting=nty,doi=false,isbn=false, url=false
]{biblatex}
\usepackage{accents}
\newlength{\dhatheight}

\newtheorem*{theorem*}{Theorem}
\newtheorem{theorem}{Theorem}[section]
\newtheorem{prop}[theorem]{Proposition}
\newtheorem{defi}[theorem]{Definition}
\newtheorem{lemma}[theorem]{Lemma}
\newtheorem{corollary}[theorem]{Corollary}
\newtheorem{example}[theorem]{Example}

\newtheorem*{corollary*}{Corollary}
\newtheorem*{prop*}{Proposition}
\newtheorem{theo}{Theorem}
\newtheorem*{lemma*}{Lemma}
\newtheorem*{conj*}{Conjecture}
\newtheorem*{defi*}{Definition}

\newcommand{\C}{\mathbb{C}}
\newcommand{\Z}{\mathbb{Z}}
\newcommand{\Pro}{\mathbb{P}}
\newcommand{\Q}{\mathbb{Q}}
\newcommand{\N}{\mathbb{N}}

\DeclareMathOperator*{\hocolim}{hocolim}
\DeclareMathOperator*{\holim}{holim}

\newcommand{\eff}{\text{eff}}
\newcommand{\et}{\text{\'et}}
\newcommand{\ch}{\text{Chow}}

\newcommand{\spc}{\text{Spec}}

\newcommand{\ho}{\text{Hom}}

\newcommand{\CH}{\text{CH}}

\theoremstyle{remark}
\newtheorem{remark}[theorem]{Remark}
\tikzcdset{
  cells={font=\everymath\expandafter{\the\everymath\displaystyle}},
}

\newcommand{\info}{{
  \bigskip
  \footnotesize

  \textsc{Institut de Math\'ematiques de Bourgogne, UMR 5584 CNRS, Universit\'e Bourgogne Franche-Comt\'e, F-21000 Dijon, France}\par\nopagebreak
  \textit{E-mail address}: \texttt{ivan-alejandro.rosas-soto@u-bourgogne.fr}
  }}


\usepackage[left=3cm,right=3.5cm,top=2cm,bottom=3.2cm]{geometry}
\author{Iv\'an Rosas-Soto}
\date{February 2024}
\title{Chow Künneth decomposition for \'etale motives}
\begin{document}
\maketitle
\begin{abstract}
In the present article we define an integral analogue of Chow-Künneth decomposition for étale motives. By using families of conservative functors we are able to establish a decomposition of the étale motive of commutative group schemes over a base and we relate to an integral étale Chow-Künneth decomposition of abelian varieties. For a projective variety $X$ of dimension $d$ over an algebraically closed field, we construct integral sub-motives $h^1_\et(X)$ and $h^{2d-1}_\et(X)$ of $h_\et(X)$.
\end{abstract}
\tableofcontents

\section{Introduction}
Let $S$ be a noetherian finite dimensional scheme and let $G/S$ a smooth commutative group scheme of finite type over $S$. According to \cite[Theorem 3.7]{AHP} in the motivic category $\text{DM}_\et(S,\Q)$  we have a decomposition of the relative motive $M_S(G)$ in the following way
\begin{align*}
    M_S(G) \xrightarrow{ \simeq } \left( \bigoplus_{n\geq 0}^{\text{kd}(G/S)} \text{Sym}^n M_1(G/S) \right)\otimes M(\pi_0(G/S)),
\end{align*}
where $M_1(G/S)$ is the 1-motive induced by the étale sheaf represented by $\underline{G/S} \otimes \Q$ and $\text{kd}(G/S):=\max\left\{2g_s+r_s \ | \ s \in S\right\}$ is the Kimura dimension, where $g_s$ is the abelian rank of $G_s$ and $r_s$ is the torus rank. The aim of this article is to see if we can lift this isomorphism to integral coefficients in $\text{DM}_\et(S,\Z)$. 

Let us recall some facts concerning the realization of $M_S(G) \in \text{DM}_\et(S,\Q)$. Consider the 1-motive associated to $G/S$, denoted by $M_1(G/S)$. As is stated in \cite[Section 5]{AHP}, the Betti realization of $M_1(G/S) \in \text{DM}\et(S,\Q)$  is given by $R_B(M_1(G/S)[-1]) = R^{2d-1}p^{an}_!\Q[2d] $. Also, we have that $R_B(M_S(G)) = p^{an}_! (p^{an})^!\Q =  p^{an}_!\Q[2d] $ and $p_!^{an}\Q(d)[2d]=\bigoplus_{i=0}^{\text{kd}(A)}\left(\bigwedge^i \mathcal{H}_1(G/S)\right)[i]$ where $\mathcal{H}_1(G/S):=R_B(M_1(G/S)[-1])$. On the other hand, for the $\ell-$adic realization of the 1-motive we have the following: consider the realization map 
\begin{align*}
    R_\ell:\text{DM}_c(S,\Q) \to D^b_c(S_\et,\Q_\ell)
\end{align*}
which is a covariant functor that is compatible with the six functor formalism on both sides. The $\ell$-adic sheaf $\mathcal{H}_1(G/S,\Q_\ell)$ is given by the torsion sheaves $G[\ell^n]$, this also holds integrally. Another important fact is related to with the étale cohomology of a connected commutative algebraic group $G$ over an algebraically closed field $k$. According to \cite[Lemma 4.1]{BSz} we have that
\begin{align*}
    H^*_\et(G,\Z/\ell) \simeq \bigwedge^{*}H^1_\et(G,\Z/\ell)
\end{align*}
for $\ell$ a prime number different from $\text{char}(k)$. The same also holds for coefficients $\Z_\ell$ and $\Q_\ell$ according to \cite[Lemma 15.2]{Mil84}.

The aim of this article is to obtain an integral analogue of the decompositon of the étale motive associated to a commutative group scheme over a base $S$ in the category $\text{DM}_\et(S,\Z)$, following the results of \cite{AHP}.

In order to obtain such decomposition, we use the \textbf{homotopy fixed points}  and \textbf{homotopy orbits} of $\mathfrak{S}_n$ of a motive $M_\et^S(X)$ which is defined as follows: knowing that $\text{DM}_\et(S,\Z)^{\otimes}$ has a structure of an $\infty$-category which is monoidal and symmetric, thus we obtain adjunctions 
 \begin{align*}
 ( \ )^{\text{triv}}: \text{DM}_\et(S,\Z)&\leftrightarrows \text{DM}_\et(S,\Z)^{B\mathfrak{S}_n}: ( \ )^{h\mathfrak{S}_n}:=\holim_{B\mathfrak{S}_n},\\    \hocolim_{B\mathfrak{S}_n}=:( \ )_{h\mathfrak{S}_n}: \text{DM}_\et(S,\Z)^{B\mathfrak{S}_n}&\leftrightarrows \text{DM}_\et(S,\Z): ( \ )^{\text{triv}}.
 \end{align*}
With this definitions, one obtains the integral analogue of \cite{AEWH} as follows:

\begin{theo}[Proposition \ref{teoprin}]
 Let $k$ be an algebraically closed field and $G/k$ a connected commutative group scheme. Then the morphism    
 \begin{align*}
      \phi_G:M_\et(G) \to \bigoplus_{i=0}^{\text{kd}(G)} \left(M_1(G)^{\otimes i}\right)^{h\mathfrak{S}_i}
 \end{align*}
 is an isomorphism in $\text{DM}_\et(k,\Z)$.
\end{theo}

If we apply the following  result which mimic the conclusion given in \cite[Lemma A.6.]{AHP}

\begin{lemma*}[Lemma \ref{consCD}]
  Let $S$ be a scheme which has finite Krull dimension and the punctual p-cohomological dimension is bounded for all prime number $p$. Then the following holds:
  \begin{enumerate}
      \item Let $M \in \textbf{DA}^\et(S,\Z)$ be a motive. Then $M$ is zero if and only if the pullbacks $i^*_{\bar{s}}M$ to any geometric point $\bar{s} \to S$ is zero. 
      \item Let $f$ be a morphism in $\textbf{DA}^\et(S,\Z)$. Then $f$ is an isomorphism if and only if the pullback $i^*_{\bar{s}}(f)$ for any geometric point   $\bar{s} \to S$ is an isomorphism.
  \end{enumerate}
\end{lemma*}

we then obtain the relative version of Proposition \ref{teoprin}:

\begin{theo}[Theorem \ref{relat}]
    Let $S$ be a scheme which has finite Krull dimension and the punctual p-cohomological dimension is bounded for all prime number $p$, and let $G$ be a connected commutative scheme over $S$. Then the morphism    
 \begin{align*}
      \phi_G:M_\et^S(G) \to \bigoplus_{i=0}^{\text{kd}(G/S)} \left(M_1(G/S)^{\otimes i}\right)^{h\mathfrak{S}_i}
 \end{align*}
 is an isomorphism in $\text{DM}_\et(S,\Z)$.
\end{theo}

We conclude the section about the decomposition of étale motives by giving a result involving the decompostion of the étale Chow motive of the product of Jacobian varieties:

\begin{theo}[Theorem \ref{prodJac}]
Let $k=\bar{k}$ be a field and consider $C_i/k$ a  projective smooth curve, for $i \in \{1,\ldots,n\}$. Then the variety $J(C_1)\times \ldots \times J(C_n)$ admits an integral Chow-Künneth decomposition. 
\end{theo}

Also, this theory gives an improvement and conditions about the existence of projectors and decomposition of integral étale motives:

\begin{theo}
 Let $X$ be a smooth projective variety of dimension $d$ over an algebraically closed field $k$. If $\text{Pic}^0(X)$ is a principally polarized variety, then there exists a decomposition of the motive $h_\et(X)$ as
 \begin{align*}
     h_\et(X)=h_\et^0(X)\oplus h^1_\et(X) \oplus h^+_\et(X) \oplus  h^{2d-1}_\et(X) \oplus h_\et^{2d}(X). 
 \end{align*}
 \end{theo}

We end the present article saying a few words about smooth conic (or quadric) bundles over a field $k$: the integral Hodge conjecture for 1-cycles holds for smooth quadric bundles $f:X \to \Pro^2_\C$ and for coninc bundles $f:X \to S$ when $S$ is smooth. If $k=\bar{k}$, then we characterize a decomposition of the étale motive of $X$ generalizing the results given in \cite{NS}.

Roughly speaking the structure of the present article is the following: we start with section 2, which is used as an introduction to the main important results and preliminaries about the triangulated category $\text{DM}_\et(S,\Z)$, concerning the different models that we can use. Section 3 we give the definition of an integral étale Chow-Künneth decomposition an present some varieties that full-fill this definition using results of \cite{RS}. Following the same spirit, we present an improved version of Manin's identity principle for étale motives. Finally, in section 5, we proof the main result of this article concerning the decomposition of the étale motive associated to a commutitative group scheme over a base $S$ with integral values and giving conditions for such decomposition to be in $\ch_\et(k)$.

\section*{Acknowledgments}

The author thanks his advisors Fr\'ed\'eric D\'eglise and Johannes Nagel for their suggestions, useful discussions and the time for reading this article. The author also thanks Giuseppe Ancona for the valuable commentaries in his report that he made about the author's thesis, which leads to this paper. This work is supported by the EIPHI Graduate School (contract ANR-17-EURE-0002) and the FEDER/EUR-EiPhi Project EITAG. We thank the French “Investissements d’Avenir” project ISITE-BFC (ANR-15-IDEX-0008) and the French ANR project “HQ-DIAG” (ANR-21-CE40-0015).

\section{Preliminaries}

We recall the definition 
of two models for the category étale of étale motives: the first one $\textbf{DA}_\et(S,\Lambda)$, uses étale sheaves without transfers, and the second one, $\textbf{DM}_\et(S,\Lambda)$, uses sheaves with transfers. For the first model we mainly use the references \cite{ayoub2014} and \cite{Ayo}; for the second one we use \cite{CD16}.

Let $\Lambda$ be a commutative ring which in this context is called the ring of coefficients. We are interested in the cases when $\Lambda= \mathbb{Z}, \ \mathbb{Q}, \ \mathbb{Z}/m$ (we will omit $\Lambda$ in the notation when $\Lambda=\mathbb Z$). 
We fix a noetherian scheme $S$ as our base scheme and we denote $\text{Sch}/S$ and $\text{Sm}/S$ the categories of schemes of finite type and smooth schemes over $S$ respectively. We denote by $\text{Sh}_\et(\text{Sm}/S,\Lambda)$ the category of étale sheaves with values in $\Lambda-$modules. 

For a given object $X$ in $\text{Sm}/S$ we denote by $\Lambda^S_\et(X)$ the étale sheaf associated to the presheaf $U \mapsto \Lambda[\ho_{\text{Sm}/S}(U,X)]$ where $ \Lambda[\ho_{\text{Sm}/S}(U,X)]$ is the free $\Lambda-$module generated by $\ho_{\text{Sm}/S}(U,X)$. 

Consider the derived category of étale sheaves $\mathbf{D}(\text{Shv}_\et(\text{Sm}/S,\Lambda))$ and denote by $\mathcal{L}$ the subcategory of the derived category of étale sheaves that contains the two complexes
\begin{align*}
    \ldots \to 0 \to \Lambda^S_\et(\mathbb{A}^1_U) \to \Lambda_\et(U) \to 0 \to \ldots
\end{align*}
and is closed under arbitrary direct sums. Here $\mathbb{A}^1:=\spc(\mathbb{Z}[t])$ and $U$ is a smooth $S-$scheme the non-zero map is induced by the projection $\mathbb{A}^1_U \to U$.

\begin{defi}
Define $\textbf{DA}^\eff_\et(S,\Lambda)$ as the Verdier quotient of $\mathbf{D}(\text{Sh}_\et(\text{Sm}/S,\Lambda))$ by $\mathcal{L}$. An object in $\textbf{DA}^\eff_\et(S,\Lambda)$ is called an effective motivic sheaf over $S$ with coefficients in $\Lambda$. The motivic sheaf $\Lambda^S_\et(X)$ is called the effective homological motive of $X$ and from now on we will denote it by $M^S_\et(X)$. 

\end{defi}

It is necessary to remark that the category $\textbf{DA}^\eff_\et(S,\Lambda)$ has the same objects of the category $\mathbf{D}(\text{Sh}_\et(\text{Sm}/S,\Lambda))$, the difference lies in the morphisms of the category since every morphism in $\mathbf{D}(\text{Sh}_\et(\text{Sm}/S,\Lambda))$ whose cone is in $\mathcal{L}$ gets inverted in $\textbf{DA}^\eff_\et(S,\Lambda)$. In particular there is an isomorphism $M^S_\et(\mathbb{A}^1_X)\to M^S_\et(X)$ for all $X \in \text{Sm}/S$ induced by $p:\mathbb{A}^1_X \to X$. Another important observation is that $\textbf{DA}^\eff_\et(S,\Lambda)$ inherits the monoidal structure of $\mathbf{D}(\text{Sh}_\et(\text{Sm}/S,\Lambda))$, which at the same time comes from the monoidal structure of $\text{Sh}_\et(\text{Sm}/S,\Lambda)$.

Let $\mathbf{L}$ be the Lefschetz motive defined as the cokernel of the inclusion $\Lambda^S_\et(\infty_S)\hookrightarrow \Lambda^S_\et(\Pro^1_S)$. The next step in the construction of the triangulated category of motivic étale sheaves is to invert the Lefschetz motive for the monoidal structure. The process used to formally invert the Lefschetz motive in \cite{Ayo} is to consider $\mathbf{L}-$spectra for the tensor product.

\begin{defi}
An $\mathbf{L}-$spectrum of étale sheaves on $\text{Sm}/S$ is a collection of étale sheaves
\begin{align*}
    E=(E_n,\gamma_n)_{n \in \mathbb{N}}
\end{align*}
where $\gamma_n:\mathbf{L} \otimes E_n \to E_{n+1}$ is a morphism of sheaves called the n-th assembly map. We call the sheaf $E_n$ the n-th level of the $\mathbf{L}-$spectrum $E$.
\end{defi}
A morphism of $\mathbf{L}-$spectra $f:E\to E'$ is a collection of morphism of sheaves $f=(f_n)_{n \in \mathbb{N}}$, where $f_n:E_n \to E_n'$ such that the diagram
 \[
  \begin{tikzcd}
\mathbf{L} \otimes E_n \arrow{r}{\text{id}\otimes f_n} \arrow{d}{\gamma_n}& \mathbf{L} \otimes E_n' \arrow{d}{\gamma'_n} \\
 E_{n+1} \arrow{r}{ f_{n+1}} &  E_{n+1}'
  \end{tikzcd}
\]
commutes for all $n \in \mathbb{N}$. We denote by $\textbf{Spt}_{\mathbf{L}}(\text{Sh}_\et(\text{Sm}/S,\Lambda))$ the category of $\mathbf{L}$-spectra.
Consider an $\mathbf{L}-$spectrum $E$. The evaluation functor $\text{Ev}_p:\textbf{Spt}_{\mathbf{L}}(\text{Sh}_\et(\text{Sm}/S,\Lambda)) \to\text{Sh}_\et(\text{Sm}/S,\Lambda) $  admits a left adjoint $\text{Sus}^p_\mathbf{L}$ given by 
\begin{align*}
    \text{Sus}^p_\mathbf{L}(K)=(\overbrace{0,\ldots,0}^{p-1\ \text{times}},K,\mathbf{L}\otimes K, \mathbf{L}^{\otimes 2}\otimes K, \ldots, )
\end{align*}

When $p=0$ the suspension functor is called the infinite suspension functor and it is denoted by $\Sigma^\infty_\mathbf{L}$. Finally, we define $\textbf{DA}^\et(S,\Lambda)$ as the Verdier quotient  of the category $\mathbf{D}(\textbf{Spt}_\mathbf{L}(\text{Shv}_\et(\text{Sm}/S,\Lambda)))$ by the smallest triangulated subcategory  $\mathcal{L}_{st}$ closed by arbitrary sums and containing the complexes
\begin{align*}
    \ldots \to 0 \to \text{Sus}^p_\mathbf{L} \Lambda^S_\et(\mathbb{A}^1_U) \to  &\text{Sus}^p_\mathbf{L}\Lambda^S_\et(U) \to 0 \to \ldots \\
    \ldots \to 0 \to \text{Sus}^{p+1}_\mathbf{L} (\mathbf{L}\otimes \Lambda^S_\et(U))& \to  \text{Sus}^p_\mathbf{L} \Lambda^S_\et(U) \to 0 \to \ldots 
\end{align*}
for all $U \in \text{Sm}/S$ and all $p \in \mathbb{N}$.

\begin{defi}
The objects in the category $\textbf{DA}^\et(S,\Lambda)$ are called motivic étale sheaves over $S$. Given a smooth $S-$scheme $X$, then $\Sigma^\infty_\mathbf{L} \Lambda^S_\et(X)$ is called the homological motive of $X$ and will be denoted by $M^S_\et(X)$. We denote $\textbf{DA}_{ct}^\et(S,\Lambda)$ the smallest triangulated subcategory of $\textbf{DA}^\et(S,\Lambda)$ closed under direct summands and containing the motives $M^S_\et(X)(-p)[-2p]:=\text{Sus}^p_\mathbf{L}\Lambda^S_\et(X)$ for $p \in \mathbb{N}$ and $X$ an $S-$scheme of finite presentation. Those motivic sheaves are called constructible.
\end{defi}

The category $\text{DM}_\et(S,\Lambda)$ is constructed in a similar way as $\textbf{DA}^\et(S,\Lambda)$, but instead of considering the whole category $\text{Shv}_\et(\text{Sm}/S,\Lambda)$ we consider the category of étale sheaves with transfers, i.e. sheaves that come from an étale presheaf that is an additive contravariant functor.

Denote as $\textbf{SmCor}(S,\Lambda)$ the category of smooth correspondences over $S$ with coefficients in $\Lambda$. The objects are the same ones of $\text{Sm}/S$, and for $U,V \in \text{Sm}/S$ the morphisms are finite $\Lambda$-correspondences from $U\to V$. Let $\textbf{Sh}_\et(\textbf{SmCor}(S,\Lambda))$ be the category of additive presheaves of commutative groups on $\textbf{SmCor}(S,\Lambda)$ whose restriction to $\text{Sm}/S$ is an étale sheaf. We call this the category of étale shaves with transfers. According to \cite[Corollary 2.1.12]{CD16} there is an adjoint pair of functors
\begin{align*}
     \gamma^*:\text{Sh}_\et(\text{Sm}/S,\Lambda)\rightleftarrows \textbf{Sh}_\et(\textbf{SmCor}(S,\Lambda)) : \gamma_*.
\end{align*}
Let $X$ be a smooth $S$-scheme. We denote by $\Lambda^{tr}(X)$ the complex of sheaves given by $c(-,X)$ the finite correspondences and let $\underline{X}$ be the sheaf  associated to $X$ defined by the presheaf
\begin{align*}
    U \mapsto \underline{X}(U)= \Lambda\text{Hom}_{\text{Sm}/S}(U,X).
\end{align*}
of commutative groups. The functor $\gamma_*$ forgets transfers and 
$\gamma^* (\underline{X})=\Z^{tr}(X)$. We continue by considering the derived category 
$D(\textbf{Sh}_\et(\textbf{SmCor}(S,\Lambda)))$ as $\textbf{Sh}_\et(\textbf{SmCor}(S,\Lambda))$ is an abelian category. After that we take the $\mathbb{A}^1$-localization of the derived category $D(\textbf{Sh}_\et(\textbf{SmCor}(S,\Lambda)))$, giving us the triangulated category of effective étale motives $\text{DM}_\et^{\text{eff}}(S,\Lambda)$. Finally, to this category we can associate a stable $\mathbb{A}^1$-derived category $\text{DM}_\et(S,\Lambda)$, the category of triangulated étale motives, by $\otimes$-inverting the Tate object $\Lambda^{\text{tr}}_S(1):=\Lambda^{\text{tr}}_S(\Pro_S^1,\infty)[-2]$. This can be obtained by applying the functor $\Sigma^\infty$.

The functor $\gamma^*$ is a left Quillen functor, thus we have its derived version
\begin{align*}
     L\gamma^*:D(\text{Sh}_\et(\text{Sm}/S,\Lambda))\rightleftarrows D(\textbf{Sh}_\et(\textbf{SmCor}(S,\Lambda))) : \gamma_*.
\end{align*}
which preserves $\mathbb{A}^1$-equivalences. With this we obtain an adjunction in the following way 
\begin{align*}
    L\gamma^*  :\textbf{DA}^\et(S,\Lambda) \rightleftarrows \text{DM}_\et(S,\Lambda): R\gamma_*
\end{align*}
If $S$ is a noetherian scheme of finite dimension, notice that by \cite[Théorème B.1]{Ayo} and \cite[Remark 5.5.9]{CD16}, the categories above mentioned are equivalent. In the context of Voevodsky motives the constructible (compact) objects are called the geometrical motives; the corresponding category is denoted by $\text{DM}^{gm}_\et(S,\Lambda)$. Also by \cite[Corollary 5.5.5]{CD16} for a quasi-excellent geometrically unibranch noetherian scheme of finite dimension $S$ the adjunctions
\begin{align*}
    \mathbf{L}\psi_! : \text{DM}_\et(S,R) \rightleftarrows \text{DM}_h(S,R):\mathbf{R}\psi^*
\end{align*}
give us an equivalence of monoidal triangulated categories.

Among the properties that we have to mention about the different models for the triangulated category of étale motives, is the one concerning conservative functors. Let $\mathfrak{TC}$ be the 2-category of triangulated categories. According to \cite[Théorème 3.9]{Ayo}, for a commutative ring $\Lambda$ the homotopic stable 2-functor $\textbf{DA}^\et(-,\Lambda): \text{Sch}/S \to \mathfrak{TC}$ is separated, this means that for any $S$-morphism $f:X \to Y$, the induced functor $f^*$ is conservative. Fixing a field $k$, let us consider a field extension $K/k$, and the induced map $p:\spc(K)\to \spc(k)$. Since $p$ is a surjective $k$-morphism, we have that 
\begin{align*}
    p^*:\textbf{DA}^\et(k,\Lambda) \to \textbf{DA}^\et(K,\Lambda)
\end{align*}
is conservative. Using the equivalence of categories $\textbf{DA}^\et(k,\Lambda) \simeq \text{DM}_\et(k,\Lambda)$ (and the same for $K$), we obtain that $p^*:\text{DM}_\et(k,\Lambda)\to \text{DM}_\et(K,\Lambda)$ is also conservative. The next examples of a conservative family: according to \cite[Proposition 3.24]{Ayo}, for a scheme $S$ such that the cohomological dimension of the residue fields is bounded, the family of functors $x^*: \textbf{DA}^\et(S,\Lambda) \to \textbf{DA}^\et(x,\Lambda)$ ,for $x \in S$, is conservative.

Now consider a noetherian scheme $S$, by \cite[Lemma A.6]{AHP}, a motive $M \in \textbf{DA}^\et(S,\Q)$ is zero if and only if the pullback to any geometric point $i_{\bar{s}}:\bar{s}\to S$ is zero. Even more, a morphism $f \in \textbf{DA}^\et(S,\Q)$ is an isomorphism if and only if $i^*_{\bar{s}}(f)$ is an isomorphism for any geometric point $\bar{s}$. This theorem can be extended to $\Z$-coefficients by imposing more restrictions on the base $S$. In order to do this, let us recall some definitions given in \cite[Définition 3.12]{Ayo}: for a prime number $p$, we define the \textit{punctual p-cohomological dimension} of a scheme $S$ as $\text{pcd}_p(S)=\sup_{s \in S}\left\{\text{cd}_p(\kappa(s))\right\} \in \N \cup \{\infty\}$, where 
$\kappa(s)$ is the residue field of a point $s \in S$.
\begin{defi}\label{gooden}
    Let $S$ be a scheme. We say that $S$ is good enough for this purposes if it has finite Krull dimension and the punctual p-cohomological dimension is bounded for every prime $p$.
\end{defi}

We can now move on to the following lemma, by mimicking the proof given for \cite[Lemma A.6.]{AHP}:

\begin{lemma}\label{consCD}
  Let $S$ be a good enough scheme. Then the following holds:
  \begin{enumerate}
      \item Let $M \in \textbf{DA}^\et(S,\Z)$ be a motive. Then $M$ is zero if and only if the pullback $i^*_{\bar{s}}M$ to any geometric point $\bar{s} \to S$ is zero. 
      \item Let $f$ be a morphism in $\textbf{DA}^\et(S,\Z)$. Then $f$ is an isomorphism if and only if the pullback $i^*_{\bar{s}}(f)$ is an isomorphism for any geometric point   $\bar{s} \to S$.
  \end{enumerate}
\end{lemma}

\begin{proof}
This follows from arguments given in \cite{AHP}. By \cite[Proposition 3.24]{Ayo} we can assume that $S=\spc(k)$ with $k$ a field. Assuming that $k$ is perfect, consider an algebraic closure $\bar{k}$. Let $N \in \textbf{DA}^\et(k,\Z)$ be a motive such that the pullback $i^*N$, with $i:\bar{k} \to k$, vanishes. Under the assumptions on $S$, 
$\textbf{DA}^\et(S,\Z)$ is compactly generated. Therefore we have to prove that all morphism $f: C \to N$ with $C$ compact vanish. Using the assumptions, $i^*(f)$ vanishes, ans according to  \cite[Lemme 3.4]{Ayo},  there exists a finite extension $K/k$ such that the pullback of $f$ vanishes.  By \cite[Théorème 3.9]{Ayo} the functor $i^*_K: \textbf{DA}^\et(k,\Z) \to  \textbf{DA}^\et(K,\Z)$ is conservative, therefore $f$ vanishes. When $k$ is not perfect, we consider a purely inseparable extension $k^i$ and note that the pullback functor $\textbf{DA}^\et(k,\Z) \to  \textbf{DA}^\et(k^i,\Z)$ is an equivalence of categories, by \cite[Proposition 6.3.16]{CD16}. The second statement follows from the first one.
\end{proof}

\section{\'Etale decomposition}

We present some applications of \cite[Theorem 1.1]{RS} to the integral decomposition of étale Chow motives. The simplest case is the one described in \cite[Appendix C]{MNP} for varieties without transcendental cohomology classes in degrees different from the dimension. For that, we give the following definition of an integral Chow-Künneth decomposition:

\begin{defi}
Let $k$ be a field and let $f:X\to k$ be a smooth projective variety, of dimension $d$. We say that $h_\et(X)$ admits an integral Chow-Künneth decomposition in $\ch_\et(k)$ if:
\begin{itemize}
    \item $h(X)$ admits a rational Chow-Künneth decomposition, see \cite[Definition 6.1.1]{MNP}, 
\begin{align*}
    h(X)\xrightarrow{\simeq }\bigoplus_{i=0}^{2d} h^i(X) \in \ch(k)_\Q,
\end{align*}
and this map is induced by a morphism $g: h_\et(X)\to M=(Y,p)$ in $\ch_\et(k)$.

\item  Consider the base change to the algebraic closure $\bar{g}:h_\et(X_{\bar{k}})\to M_{\bar{k}}$. For every prime number $\ell \neq \text{char}(k)$, the induced map $\rho_\ell(\bar{g}):R\bar{f}_*(\Z/\ell)\to M_{\bar{k}}/\ell \in D(\bar{k}_\et,\Z/\ell)$ is an isomorphism and $\rho_\ell(\bar{p})=p_1+\ldots+p_{2d}$ with the following conditions
\begin{align*}
    p_i\circ p_j =\begin{cases}
        p_i \text{ if }i=j\\
        \ 0 \text{ if }i\neq j,
    \end{cases}\quad 
    \rho(\bar{g})^{-1}\circ p_i (M_{\bar{k}}/\ell)=R^i\bar{f}_*(\Z/\ell) \text{ for all } i.
\end{align*}

\end{itemize}
\end{defi}

This is nothing but a direct translation of the conservativity properties of the family of functors associated to the change of coefficients in \cite[Proposition 5.4.12]{CD16} and combining the results about conservativity \cite[Théorème 3.9]{Ayo} and continuity \cite[Proposition 6.3.7]{CD16} applied to $\bar{k}=\varprojlim k_i$ where $k_i$ runs over the finite fields extensions of $k$.

\begin{prop}
    Consider a field $k$ of finite cohomological dimension and let $h_\et(X) \in \ch_\et(k)$. Then $h_\et(X_{\bar{k}})$ has an integral Chow-Künneth decomposition if and only if there exists a field extension $K/k$ such that $h_\et(X_K)$ has an integral Chow-Künneth decomposition.
\end{prop}

\begin{proof}
For simplicity, up to tensoring with Lefschetz motive, which is a direct summand of a geometric motive, we may assume that $M=(X,p)$. If  $M_{\bar{k}}$ has an integral Chow-Künneth decomposition then the result is trivial. 

Conversely, assume that there exists a Chow-Künneth decomposition for some field extension $K/k$. For the rational part we invoke \cite[Proposition 1.5]{Via}. For the torsion part, let $\ell \neq \text{char}(k)$ be a prime number and consider the field extension $K/k$. Consider the morphism of $s:\spc(\bar{K})\to \spc(\bar{k})$. Let us assume that $h_\et(X_K)$ has a Chow-Künneth decomposition, thus we have that $\rho_\ell(\bar{g_K}): R\bar{f}_K(\Z/\ell) \to M_{\bar{K}}/\ell$ with the above properties for $\rho_\ell(\bar{p}_{K})$. The induced functor $s^*:D(\bar{k}_\et,\Z/\ell) \to D(\bar{K}_\et,\Z/\ell)$ is an equivalence of categories, hence $\rho_\ell(\bar{g}):R\bar{f}_*(\Z/\ell)\to M_{\bar{k}}/\ell \in D(\bar{k}_\et,\Z/\ell)$  is an isomorphism, thus we have the same results for $\rho_\ell(\bar{p})$ and conclude the proof. 

\end{proof}

Let us comeback to the decomposition of étale Chow-Künneth decomposition of complex varieties:

\begin{prop}\label{PropM}
Fixing $k=\C$, let $X$ be a smooth projective complex variety of dimension $d$ such that the groups $H^i_B(X, \Q)$ are algebraic for all $i \neq d$. Then $h_\et(X)$ admits an integral Chow-Künneth decomposition in $\ch_\et(\C)$.
\end{prop}

\begin{proof}
We will use the equivalence given in \cite[Theorem 1.1]{RS} and \cite[Appendix C]{MNP}. Let us start by saying that according to \cite[Theorem 1.1.a]{RS} the map Lichtenbaum cycle class map
$$c_L^{m,n}:H^m_L(X,\Z(n)) \to H^m_B(X,\Z(n))$$
restricted to the torsion subgroup $H^m_L(X,\Z(n))_{\text{tors}} \to H^m_B(X,\Z(n))_{\text{tors}}$ is surjective. With this in mind we consider that the groups $H^i_B(X,\Z)$ are torsion free and then Poincaré duality holds, i.e. the pairing 
\begin{align*}
    H^i_B(X,\Z) \otimes H^{2d-i}_B(X,\Z) &\to \Z \\
    (\alpha,\beta) &\xrightarrow{ \cup } \alpha \cup \beta
\end{align*}
is perfect. By \cite[Theorem 1.1]{RS} we have  $H^{2i}_B(X, \Q)$ is algebraic if and only if $H^{2i}_B(X,\Z)$ is L-algebraic, thus there exists a set of cycles which are send to the generators $\left\{e^{2i}_j\right\}_{1\leq j \leq b_{2i}(X)}$ of $H^i_B(X,\Z)$ and notice that by Poincaré duality we have a dual basis $\left\{\hat{e}^{2(d-i)}_j\right\}_{1\leq j \leq b_{2(d-i)}(X)}$ for the dual of $H^{2i}_B(X,\Z)$. Let us remark that we have the following
\begin{align*}
    e^{2i}_j \cup \hat{e}^{2(d-i)}_l =\begin{cases} 0 \text{ if }j \neq l \\
    1\text{ if }j = l
    \end{cases}
\end{align*}

By hypothesis, there exists L-algebraic cycles $\left\{\alpha_j^{i}\right\}_{1\leq j \leq b_{2i}(X)} \subset \CH^i_L(X)$ and 
$$\left\{\hat{\alpha}_l^{d-i}\right\}_{1\leq l \leq b_{2(d-i)}(X)}\subset \CH^{d-i}_L(X)$$ such that 
\begin{align*}
    c^i_L(\alpha_j^{i})= e_j^{2i}, \hspace{4mm} c^{d-i}_L(\hat{\alpha}_l^{d-i})= \hat{e}_l^{2(d-i)}
\end{align*}
for all $1\leq j \leq b_{2i}(X)$ and $1\leq l \leq b_{2(d-i)}(X)$. Due to the compatibility of the cycle class map with intersection products we have that 
\begin{align*}
    \alpha_j^{i} \cdot \hat{\alpha}^{d-i}_l =\begin{cases} 0 \text{ if }j \neq l \\
    1\text{ if }j = l.
    \end{cases}
\end{align*}

Let us define the elements 
\begin{align*}
    p_{2i,j}=\alpha_j^{i}\times \hat{\alpha}_j^{d-i} \hspace{4mm} q_{2i,j}=\hat{\alpha}_j^{d-i} \times \alpha_j^{i}
\end{align*}
and note that
$p_{2i,j}=q^t_{2i,j}$. Even more, these are orthogonal projectors. For $i<d$, define the projectors 
\begin{align*}
p_{2i}(X):=\sum_{1\leq j \leq b_{2i}(X)} p_{2i,j} \hspace{4mm} p_{2(d-i)}(X):=\sum_{1\leq j \leq b_{2i}(X)} q_{2i,j}
\end{align*}
and for $2i-1\neq d$ we put $p_{2i-1}(X)=0$. The remaining part should involve torsion classes. As the groups $H^{2j+1}_B(X,\Z)$ are  torsion for all $j \in \N$, the groups $H^{2k+1}_B(X\times X,\Z)$ are torsion for all $k \in \N$ by the Künneth formula, this implies that all intermediates Jacobians $J^{k+1}(X\times X)$ vanish for all $k \in \N$. Combining \cite[Theorem 1.1.b]{RS} and \cite[Proposition 3.1.5]{roso23} we obtain an isomorphism $\CH^k_L(X \times X)_{\text{tors}} \xrightarrow{ \simeq} H^{2k}_B(X\times X,\Z(k))_{\text{tors}}$ for all $k \in \N$, so in particular for the degree $k=d$. Consider that we have the diagonal element $\Delta$ and let us denote the torsion free part as $\Delta_{\text{tf}}= \sum_{j=0}^{2d} p_i(X) $ and consider $\Delta_{\text{tors}}=\Delta-\Delta_{\text{tf}}$. As this element $\Delta_{\text{tors}}  \in H^{2d}_B(X\times X,\Z)$, then it has a unique preimage in $\CH^d_L(X \times X)_\text{tors}$, which is denoted as $\Delta_{\text{tors}}$ again, thus we have the following decomposition of the diagonal
\begin{align*}
    \Delta = \sum_{j=0}^{2d} p_i(X) + \Delta_{\text{tors}}.
\end{align*}
Since the isomorphism $\CH^k_L(X \times X)_{\text{tors}} \xrightarrow{ \simeq} H^{2k}_B(X\times X,\Z)_{\text{tors}}$ is an isomorphism for all $k$, the projectors in $H^{2k}_B(X\times X , \Z)_{\text{tors}}$ can be lifted to $\CH^k_L(X \times X)$.


\end{proof}

\begin{example}
\begin{enumerate}
\item Let $X$ be a smooth complex complete intersection in projective space. As all the cohomology groups are algebraic and torsion free, we have a decomposition of étale integral motives as follows:
\begin{align*}
    h_\et(X)\simeq \mathbf{1} \oplus  \mathbb{L}\oplus \ldots \oplus h^d_\et(X) \oplus \ldots \oplus \mathbb{L}^d.
\end{align*}
where $\mathbb{L}$ is the Lefschetz motive and  $h^d_\et(X)=(X,p^{\et}_d(X),0)$ with 
\begin{align*}
   p^{\et}_d(X)=\Delta-\sum_{i=0,2i\neq d}^{2d} p_i^{\et}(X). 
\end{align*}

\item Let $X$ be a smooth K3 surface. For such $X$ we have the following isomorphisms
    \begin{align*}
        H^0(X,\Z) \simeq H^4(X,\Z) \simeq \Z, \ H^1(X,\Z) \simeq H^3(X,\Z) = 0,  \ \ H^2(X,\Z)\simeq \Z^{22}
    \end{align*}
    and $\text{Pic}(X)=\Z^{\rho(X)}$, with $\rho(X)$ the Picard rank of $X$ and $0\leq \rho(X)\leq 20$. Since the cohomology is torsion free, we apply Proposition \ref{PropM} to obtain a decomposition of the étale motive
    \begin{align*}
        h_\et(X)\simeq h^0_\et(X)\oplus h^2_\et(X) \oplus h^4_\et(X). 
    \end{align*}
    
    \item Let $S$ be an Enriques surface. As $H^i(\mathcal{O}_S)=0$ for $i=1,2$ we have an isomorphism $\text{Pic}(S)\to H^2_B(S,\Z)\simeq \Z^{10} \oplus \Z/2$ while the other cohomology groups are characterized by
    \begin{align*}
        H^0(S,\Z)=H^4(S,\Z)=\Z, \ H^1(S,\Z)=0 \ \text{ and } \ H^3(S,\Z)=\Z/2.
    \end{align*}
as we can lift the torsion free part, we have to care about the torsion part of the cohomology. By Künneth formula, we have that $H^5_B(S\times S,\Z)\simeq (\Z/2)^{\oplus 23}$ and $H^3_B(S\times S,\Z)\simeq \Z/2 \oplus \Z/2$ thus we conclude that the intermediate Jacobians $J^2(S\times S)=0$ and $J^3(S \times S)=0$ vanish. Combining \cite[Proposition 5.1]{RS} and \cite[Proposition 3.1.4]{RoSo} we have an isomorphism $\CH^2_L(S\times S)_{\text{tors}} \xrightarrow{\simeq} H^4_B(S\times S,\Z(2))_{\text{tors}}$ which acts as the identity on the torsion part.
\item For a Calabi-Yau threefold $X$ (for example a quintic threefold) $X$ the Betti numbers are $h^1(X)=h^5(X)=0$ and $h^0(X)=h^2(X)=h^4(X)=h^6(X)=1$, thus we obtain a decomposition of the motive $h_\et(X)$ as
\begin{align*}
    h_\et(X) \simeq \mathbf{1} \oplus \mathbb{L} \oplus h_\et^3(X) \oplus \mathbb{L}^2 \oplus \mathbb{L}^3.  
\end{align*}
\end{enumerate}
\end{example}


\section{Improved Manin version for the étale case}

Let us consider the functor $F^i$ defined as follows $F^i:\ch_\et(k) \to \Z\text{-mod}$, $M\mapsto F^i(M):=\ho_{\ch_\et(k)} (\mathbb{L}^i,M)$, with $M$ of the form $M=(X,p,m)$, and consider the $\Z$-graded functor $F:=\oplus_{i \in \Z} F^i: \ch_\et(k) \to \Z\text{-modGr}$. By definition of $F^i$ we have that
\begin{align*}
    F^i(M)&=\ho_{\ch_\et(k)} (\mathbb{L}^i,M)\\
    &=p \circ \CH^{i+m}_\et(\spc(k)\times X) \\
    &\simeq p_*\CH^{i+m}_\et(X)=\CH^i_\et(M).
\end{align*}
By definition $F$ is an additive functor, and by duality 
\begin{align*}
    \ho_{\ch_\et(k)}((X,p,m),(Y,q,n))&=q \circ \text{Corr}_\et^{n-m}(X,Y) \circ p\\
    &=q \circ \CH_\et^{n-m+d_X}(X\times Y) \circ p \\
    &\simeq F^0(M \otimes N^\vee).
\end{align*}
Notice that $N$ is a sub-motive of $h(Y)\otimes \mathbb{L}^n$ for some $Y \in \text{SmProj}_k$ and $n \in \Z$ and by duality $F^0(M\otimes h(Y)\otimes \mathbb{L}^n ) \simeq F^{-n}(M\otimes h_\et(Y) )$. For a fixed $M \in \ch_\et(k)$ define the following functor
\begin{align*}
    \omega_M:\text{SmProj}_k^{op}&\to \Z\text{-modGr} \\
    Y &\mapsto \omega_M(Y):=F(M\otimes h_\et(Y))
\end{align*}
then Yoneda embedding implies that the functor 
\begin{align*}
    \omega:\ch_\et(k)&\to \Z\text{-modGr}^{\text{SmProj}_k^{op}} \\
    M &\mapsto \omega_M
\end{align*}
is fully faithful. Hence we recover the classical Manin principle but in the étale setting!
\begin{prop}\label{maniet}[Manin's identity principle]
Let $f$, $g: M\to N$ be morphism of  étale motives then:
\begin{enumerate}
    \item $f$ is an isomorphism if and only if the induced map 
    \begin{align*}
        \omega_f(Y):\omega_M(h(Y))\to \omega_N(h(Y)) 
    \end{align*}
    is an isomorphism for all $Y \in \text{SmProj}_k$ and $f=g$ is and only if $\omega_f(Y)=\omega_g(Y)$ for all $Y \in \text{SmProj}_k$.
    \item A sequence
    \begin{align*}
        0\to M_1 \xrightarrow{f}  M_2 \xrightarrow{g} M_3 \to 0 
    \end{align*}
    is exact if and only if, for every $Y \in \text{SmProj}_k$ the sequence 
    \begin{align*}
        0\to \omega_{M_1}(h(Y)) \xrightarrow{  \omega_f(h(Y))} \omega_{M_2}(h(Y)) \xrightarrow{\omega_g(h(Y))} \omega_{M_3}(h(Y)) \to 0 
    \end{align*}
\end{enumerate}
\end{prop}
\begin{proof}
This properties is a consequence of the faithfulness of the functor $\omega$ and the fact that fully faithful functor reflects monic, epi and isomorphisms.
\end{proof}

The following isomorphisms in $\ch_\et(k)$ are obtained as a consequence of \cite[Lemma 2.8]{roso23} about the structure of étale motivic cohomology groups: we can obtain decomposition for motives of a projective bundle, blow-ups with smooth center and flag varieties. 

\begin{example}
\begin{enumerate}
    \item Consider $E$ a locally free sheaf of rank $(n+1)$ over $X$, and $\pi:\Pro_X(E)\to X$ its associated projective bundle. Then 
\begin{align*}
    \CH^i_\et(\Pro_X(E)) \simeq \bigoplus_{j=0}^{n}\CH^{i-j}_\et(X) .
\end{align*}
Since this isomorphism is functorial with respect to base change, for all $Y \in \text{SmProj}_k$ we have an isomorphism $ \CH^i_\et(Y\times \Pro_X(E)) \simeq \bigoplus_{j=0}^{n}\CH^{i-j}_\et(Y\times X)$, therefore we have a decomposition of the motive of $\Pro_X(E)$ as 
\begin{align*}
    h_\et(\Pro_X(E))\simeq \bigoplus_{i=0}^n h_\et(X)(-i).
\end{align*}

\item Consider $Y=\text{Bl}_Z X$ the Blow-up of $X \in \text{SmProj}_k$ along a smooth sub-scheme $Z$ of codimension $(d+1)$. Since the isomorphism described in \cite[Lemma 2.8]{roso23} is functorial with respect to base change, then we have a decomposition of the motive of $Y$ as follows
\begin{align*}
    h_\et(Y)\simeq h_\et(X) \oplus \bigoplus^m_{i=1}h_\et(Z)(-i).
\end{align*}

\item Let $S$ be a smooth $k$-scheme and let $X\to S$ be a flat morphism of relative dimension $n$ such that $X$ has a decomposition in smooth projective varieties $X=X_p\supset X_{p-1}\supset \ldots \supset  X_{0}\supset X_{-1}=\emptyset $ with $X_i-X_{i-1}\simeq \mathbb{A}_S^{n-d_i}$ for some $d_i \in \Z$. Since the characterization of the étale Chow groups of $X$ given in \cite[Lemma 2.8]{roso23} is functorial with respect to base change $S \to S\times Y$, then 
\begin{align*}
    h_\et(X)\simeq \bigoplus_{i=0}^p h_\et(S)(d_i).
\end{align*}
\end{enumerate}
\end{example}

Let us study the behaviour of the category of étale Chow motives under a field extension. For that we use the theory of étale Chow groups. Consider a field $k$ and let $\Omega$ be a universal domain of $k$. Let $M=(X,p)$ be an étale Chow motive and assume that $\CH_\et^i(M_\Omega)=0$ for all $i\geq 0$. Let us say come facts about Lichtenbaum cohomology:

\begin{prop}\label{chafie}
    Let $k$ be a field, $X$ a smooth projective $k$-scheme and $K$ a field extension of $k$. Setting an integer $i\geq 0$:
    \begin{enumerate}
        \item If $k$ is a field with $k = \bar{k}$ and $K=\bar{K}$ then the map $\CH^i_L(X)\to \CH^i_L(X_K)$ induced by the base change is injective.
        \item If $K$ is a finite purely inseparable extension then the maps $\CH_L^i(X)\to \CH_L^i(X_K)$ and $\CH_L^i(X_K)\to \CH_L^i(X)$ are isomorphisms.
    \end{enumerate}
\end{prop}

\begin{proof}
 First, let $k$ be a perfect field with $k = \bar{k}$ and consider a field extension $K$ which is again algebraically closed. By the smooth base change, we have that $H^{m}_\et(X,\mu_{\ell^r}^{\otimes n}) \to H^{m}_\et(X_K,\mu_{\ell^r}^{\otimes n})$ is an isomorphism when $\ell$ is prime to the characteristic of $k$ and then so it is the morphism $ H^{m}_\et(X,\Q_\ell/\Z_\ell(n)) \to H^{m}_\et(X_K,\Q_\ell/\Z_\ell(n))$ and from the commutative diagram with exact rows
 \begin{equation*}
  \begin{tikzcd}
0 \arrow{r} & H^{m}_L(X,\Z(n)) \otimes \Q_\ell/\Z_\ell \arrow{r} \arrow{d}& H^{m}_\et(X,\Q_\ell/\Z_\ell(n)) \arrow{r}\arrow{d}{\simeq} & H^{m+1}_L(X,\Z(n))\{\ell\} \arrow{r}\arrow[d]& 0 \\
0 \arrow{r} & H^{m}_L(X_K,\Z(n)) \otimes \Q_\ell/\Z_\ell \arrow{r} & H^{m}_\et(X_K,\Q_\ell/\Z_\ell(n)) \arrow{r} & H^{m+1}_L(X_K,\Z(n))\{\ell\} \arrow{r} & 0
  \end{tikzcd}
\end{equation*}
 we conclude that $H^{m}_L(X,\Z(n)) \otimes \Q_\ell/\Z_\ell \to H^{m}_L(X_K,\Z(n)) \otimes \Q_\ell/\Z_\ell$ is an injective morphism. Recalling that for a separably closed field as in our case we have that for $m \neq 2n+1$ an isomorphism $H^m_L(X,\Z(n))\{\ell\}\simeq H^{m-1}_\et(X,\Q_\ell/\Z_\ell(n))$. 

Let $A^{m,n} =  H^m_L(X,\Z(n))$ and $A^{m,n}_K =  H^m_L(X_K,\Z(n))$, then we have that following commutative diagram 
\begin{equation*}
  \begin{tikzcd}
0 \arrow{r} & A^{m,n} _\text{tor} \arrow{r} \arrow{d}{\simeq} & A^{m,n} \arrow{r}\arrow{d}& A^{m,n} \otimes \Q \arrow{r}\arrow[d, hookrightarrow]& A^{m,n}  \otimes \Q/\Z\arrow{r} \arrow[d, hookrightarrow]& 0 \\
0 \arrow{r} & A^{m,n}_{K,\text{tor}}\arrow{r} & A^{m,n}_K\arrow{r} & A^{m,n}_K \otimes \Q  \arrow{r}& A^{m,n}_K \otimes \Q/\Z\arrow{r} & 0
  \end{tikzcd}
\end{equation*}
the arrow $A^{m,n}\otimes \Q \to A^{m,n}_K\otimes \Q$ is an injection by classical arguments, therefore $A^{m,n} \to A^{m,n}_K$ is an injective map as well. 

For the second part the argument goes in the same direction. The isomorphism for the torsion part is a consequence of the map $X_K \to X$ which is finite surjective radiciel (see \cite[Proposition 5.7.1]{Fu15}), therefore
$H^{m-1}_\et(X,\mu_{\ell^r}^{\otimes n}) \to H^{m-1}_\et(X_K,\mu_{\ell^r}^{\otimes n})$ is an isomorphism. The isomorphism of the torsion free part is a consequence of \cite[Lemma 1.2]{Via}. We then conclude as in the previous case.
\end{proof}
For the moment, let us mention an analogue of \cite[Theorem 6.8]{Kimu}.

\begin{defi}\label{defsurj}
Let $f:M \to N$ be a morphism of étale Chow motives. We say that $f$ is a surjective morphism if for all $Z \in \text{SmProj}_k$ the induced map
\begin{align*}
    (f\otimes \text{id}_Z)_*: \CH^n_\et(M\otimes h(Z)) \to \CH^n_\et(N\otimes h(Z)) 
\end{align*}
is surjective for all $n$.
\end{defi}

\begin{lemma}\label{lieb}[Lieberman's lemma]
Let $X,Y,Z$ and $W\in \text{SmProj}_k$. Consider $f \in \text{Corr}_\et(X,Y)$, $\alpha \in  \text{Corr}_\et(X,Z)$ and $\beta \in \text{Corr}_\et(Y,W)$. Then $(\alpha \times \beta)_*(f)=\beta \circ f \circ \alpha^t$.
\end{lemma}

\begin{proof}
We follow the proof of \cite[Lemma 2.1.3]{MNP}.
\end{proof}

\begin{lemma}\label{kim6.8}
Let $f:M=(X,p,m)\to N=(Y,q,n)$ be a morphism of étale Chow motives. The following conditions are equivalent:
\begin{enumerate}
    \item $f$ is surjective.
    \item There exists a right inverse $g:N\to M$ i.e. $f\circ g = \text{id}_N$.
    \item $q=f \circ s$ for some $s \in \text{Corr}_\et^0(Y,X)$.
\end{enumerate}
\end{lemma}

\begin{proof}
(1. $\Longrightarrow$ 2.) For this implication, we use Lieberman's lemma (see Lemma \ref{lieb}) for étale correspondences. Assuming point 1. consider the particular case $Z=Y$ and $q^t \in \text{Corr}^0_\et(Y,Y)$. By Lieberman's lemma $q^t=(q\times \text{id}_Y)_*\text{id}_Y$ then $q^t \in \CH_\et^*(N\otimes h(Y))$. By assumption there exists an element $r\in \CH_\et^*(M\otimes h(Y)) \subset \CH_\et^*(X\times Y) $ such that $(f\times \text{id}_Y)_* r =q^t$, and again by Lieberman $r \circ f^t = q^t$. Take $g= p \circ r^t\circ q$.

(2. $\Longrightarrow$ 1.) As $f\circ g = \text{id}_N$ after base change using $Z \in \text{SmProj}_k$ we obtain that $(f\times \text{id}_Z)_* \circ (g\times \text{id}_Z)_* = \text{id}_{N\otimes Z}$. Therefore $(f\times \text{id}_Z)_*: \CH^n_\et(M\otimes h(Z)) \to \CH^n_\et(N\otimes h(Z)) $ is surjective. 

(2. $\Longrightarrow$ 3.) For that just take the $s$ as the correspondence associated to $f\in\text{Corr}_\et^0(Y,X)$.

(3. $\Longrightarrow$ 2.) Consider the morphism defined by $g=p\circ s \circ q$.
\end{proof}

Now again, we get an étale analogue of \cite[Lemma 3.2]{Via}:

\begin{prop}
Let $f:M \to N$ be a morphism of étale motives defined over an algebraically closed field $k$:
\begin{enumerate}
    \item Assume that for some field extension $K$ (with $K=\bar{K}$) the map $(f_K)_*:\CH^i_\et(M_K) \to \CH^i_\et(N_K)$ is injective. Then $f_*:\CH^i_\et(M) \to \CH^i_\et(N)$ is injective.
    \item Assume that for some field extension $K$ (with $K=\bar{K}$) the map $(f_K)_*:\CH^i_\et(M_K) \to \CH^i_\et(N_K)$ is surjective. Then $f_*:\CH^i_\et(M) \to \CH^i_\et(N)$ is surjective.
\end{enumerate}
\end{prop}

\begin{proof}
The first statement follows from the commutative diagram
\begin{equation*}
  \begin{tikzcd}
\CH^i_\et(M) \arrow{r} \arrow{d}{f_*} & \CH^i_\et(M_K) \arrow{d}{(f_K)_*}  \\
\CH^i_\et(N) \arrow{r} & \CH^i_\et(N_K) 
  \end{tikzcd}
\end{equation*}
and the fact that $\CH^i_\et(X)\to \CH^i_\et(X_K)$ is an injection by 
Proposition \ref{chafie}. For the surjectivity, notice that under assumptions about the base field, the map $ H^{m-1}_\et(X,\Q_\ell/\Z_\ell(n)) \to H^{m-1}_\et(X_K,\Q_\ell/\Z_\ell(n))$ is an isomorphism for every bidegree, therefore if the map after tensor with the rational is surjective ( which is the result of \cite[Lemma 3.2]{Via}), we then obtain that the map is surjective from a similar argument of Proposition \ref{chafie}.
\end{proof}

\begin{lemma}
    Let $M=(X,p,m)$ be an étale motive over an algebraically closed field $k$. Then $M=0$ if and only if $M_K=0$ for some field extension $K$.
\end{lemma}
\begin{proof}
   This is a direct consequence of the fact for a field extension $p:K \to k$ the induced functor $p^*:\text{DM}_\et(k,\Z) \to \text{DM}_\et(K,\Z)$ is conservative.
\end{proof}

\begin{remark}
    Notice that \cite[Proposition 1.3]{Via} is also a direct consequence of the condition of separateness of $\text{DM}_\et(k,\Q)$ and the ful-faithful embedding $\ch(k)^{\text{op}}\to\text{DM}(k,\Q)\simeq \text{DM}_\et(k,\Q)$.
\end{remark}

In this subsection we aim to obtain an analogue of \cite[Lemma 1.1]{Huy} for the category $\ch_\et(k)$. Roughly speaking, this result is an improved version of Manin's principle, but only when one works over an algebraically closed field. 

Manin principle says that a morphism $f: M\to N$ between Chow motives is an isomorphism if and only if the associated map $(f\times \text{id}_Z)_*: \CH^*(M \otimes h(Z))_\Q \to \CH^*(N \otimes h(Z))_\Q$ is an isomorphism for every smooth projective variety $Z$. There are few cases where the structure of the Chow groups are maintained in an easy way, such as projective bundles or blow-ups, but in general it is not an easy task to obtain this property. Recalling that a universal domain $\Omega$ over $k$ is an algebraically closed field extension of infinite transcendence degree (for example $k=\bar{\Q}$ and $\Omega=\C$), the improved Manin priciple states the following:

\begin{theorem}[{\cite[Lemma 1.1]{Huy}}]\label{imanin}
Consider an algebraically closed field $k$. Let $f:M\to N$ be a morphism in the category $\ch(k)_\Q$. Then $f$ is an isomorphism of motives in $\ch(k)_\Q$ if and only if for $\Omega$ a universal domain over $k$, the induced map $(f_\Omega)_*:\CH^*(M_\Omega)_\Q\to\CH^*(N_\Omega)_\Q$ given by the base change $f_\Omega: M_\Omega \to N_\Omega$, is bijective.
\end{theorem}

Therefore for an algebraically closed field, is not necessary to test a morphism indexed by the objects in $\text{SmProj}_k$, only for a huge field extension of $k$. The improved version of Manin principle is a direct consequence of the results \cite[Lemma 1]{GG}, \cite[Theorem 3.18]{Via} and \cite[Lemma 2.4]{BoPe}. 
\begin{example}
    Consider a conic bundle $X \to \Pro^2_{\bar{Q}}$, we have that the Chow groups of $X$ are characterized by
    \begin{align*}
        \CH^0(X)_\Q\simeq \CH^3(X)_\Q \simeq \Q, \ \CH^1(X)\simeq \Q \oplus \Q  \text{ and } \CH^2(X)\simeq \Q \oplus \Q \oplus \text{Prym}(\bar{C}/C)_\Q, 
    \end{align*}
so for this case we can recover the motivic decomposition of $X$ obtained in \cite{NS}. In this context, $C$ is called the discriminant curve of $X$, $\sigma_C:\bar{C}\to C$ is a double covering and $\text{Prym}(\bar{C}/C)$ is the Prym variety.
\end{example}

In the following, we will present the analogue of \cite[Lemma 1.1]{Huy} for the category $\ch_\et(k)$. To obtain that, we will prove the analogue of \cite[Lemma 1]{GG}:

\begin{lemma}\label{lemmM0}
    Let $M=(X,p,m)$ be an étale Chow motive defined over an algebraically closed field $k$. Let $\Omega$ be a universal domain of $k$  and assume that $\CH_\et^i(M_\Omega)=0$ for all $i\geq 0$. Then $M\simeq 0$ in $\ch_\et(k)$.
\end{lemma}

\begin{proof}
    We proceed with similar arguments as in \cite[Lemma 1]{GG}. Consider $Y \in \text{SmProj}_k$ and let $i:Z\hookrightarrow Y$ be a smooth closed immersion of codimension $c_Z$ and let $U:= Y-Z$ be the open complement
\begin{align*}
     \ldots \to \CH^{i-c_Z}_\et(X\times Z) \to \CH^i_\et(X\times Y)\to \CH^i_\et(X\times U)\to \text{Br}^{i-c_Z}(X\times Z)\to \ldots
\end{align*}
now take the direct limit  over opens $U\subset Y$ we obtain that 
\begin{align*}
     \ldots \to \bigoplus_{Z \subset Y} \CH^{i-c_Z}_\et(X\times Z) \to \CH^i_\et(X\times Y)\to \varinjlim_{U \subset Y} \CH^i_\et(X\times U)\to \bigoplus_{Z \subset Y} \text{Br}^{i-c_Z}(X\times Z)\to \ldots
\end{align*}
since we have the isomorphism $ \varinjlim_{U \subset Y} \CH^i_\et(X\times U) \simeq \CH^i_\et(X_{k(Y)})$ and consider the morphism (defined through the action of correspondences) $p\otimes Y:\CH^i_\et(X\times Y)\to \CH^i_\et(X\times Y)$ defined as follows
\begin{align*}
    (p\otimes Y )(\alpha):=(\text{pr}_{23})_*\left( \text{pr}_{13}^*(\alpha)\times p \cdot \Gamma_{\text{pr}_{12}} \right) 
\end{align*}
where $\text{pr}_{ij}:X\times X \times Y \to X_i\times X_j$ and $\Gamma_{\text{pr}_{12}}$ is the graph of the projection morphism. We apply the morphisms  $\bigoplus_Z p\otimes Z$, $p \otimes Y$ and $p\otimes k(Y)$ we then obtain the following exact sequence
\begin{align*}
    \bigoplus_{Z \subset Y} \text{im}(p\otimes Z)\to \text{im}(p\otimes Y) \to \CH^i_\et(M_{k(Y)}) \to \bigoplus_{Z \subset Y}\text{im}(p\otimes Z)_{-1}.
\end{align*}
Notice the following facts about the étale Chow groups of the motive $M$:
\begin{itemize}
    \item If $Y$ is irreducible with $\dim(Y)=0$, then $\text{im}(p\otimes Y)=\CH^i_\et(M)$, and consider  $\Omega$ a field extension of $k$ which is algebraically closed. Then we have that $\CH^i_\et(M)\to \CH^i_\et(M_\Omega)$ is injective, so by the hypothesis $\CH^i_\et(M)=0$  for all $i\geq 0$.
    \item By induction, assume that for $Z$ of dimension $0,\ldots,n-1$ we have that $\text{im}(p\otimes Z)$ vanish, then $\text{im}(p \otimes Y) $ injects in $\CH^i_\et(M_{k(Y)})$ by the localization sequence. By \cite[Lemma 1]{GG}, the action of $p \otimes k(Y)$ over the torsion free part of $\CH_\et^i(X_{k(Y)})$, then $\CH^i_\et(M_{k(Y)})\simeq (p \otimes k(Y))_* H^{2i-1}_\et(X_{k(Y)},(\Q/\Z)'(i))$. 
    
    To conclude we will use a specialization argument. Consider a open subset $U \subset Y$ and consider the motive $(X_U,p_U)$. Now let $u$ be a closed point of $U$ therefore we can define the regular embedding $j_u:u \hookrightarrow U$. Notice that the closed fibers of $U$ are isomorphic to  $(X,p)$ over $k$. Since the specialization map commutes with products, pull-backs and pushforwards, we obtain that the projector $p \otimes U$ acts as zero over $\CH^{i}_\et(X \times U)$, therefore we conclude that $p\otimes k(Y)$ acts as zero over $\CH^{i}_\et(X_{k(Y)})$. Finally, we conclude that $\CH^i_{\et}(M_{k(Y)})=0$ for all integer $i\geq 0$.


\end{itemize}
Since we have that $\CH^i_{\et}(M_{k(Y)})=0$ and 
  $\text{im}(p \otimes Y)$ injects into  $\CH^i_{\et}(M_{k(Y)})$ for all $Y \in \text{SmProj}_k$, since $\text{im}(p \otimes Y)\simeq \CH^i_\et(M\otimes h(Y))$ by the Manin principle for étale motives we can conclude that $M=0$.
\end{proof}

Along with Definition \ref{defsurj} and Lemma \ref{kim6.8} we obtain the analogue of \cite[Theorem 3.18]{Via} for étale motivic cohomology:

\begin{lemma}\label{surj}
    Let $f:M\to N$ be a morphism of motives over $k$ with $k=\bar{k}$,  $k \hookrightarrow \Omega$, with $\Omega$ an universal domain, such that the induced morphism $\left(f_\Omega\right)_*:\CH^*_\et(M_\Omega)\to \CH^*_\et(N_\Omega)$ is surjective. Then $f$ is surjective.
\end{lemma}

\begin{proof}
    Let $f$ be a morphism of motives over $k$ and let $\Omega$ be a universal domain such that $k \hookrightarrow \Omega$. Consider $Z \in \text{SmProj}_k$, we will prove that the morphism $f\otimes Z : \CH^*_\et(X\times Z) \to \CH^*_\et(Y\times Z)$ has the same image as $q\otimes Z : \CH^*_\et(Y\times Z) \to \CH^*_\et(Y\times Z)$, for all $Z$. In order to prove this, we proceed by induction over the dimension of $Z$.

    If $\text{dim}(Z)=0$ then the result is clear. So let us assume that works for $\text{dim}(Z)\leq n$. Consider the following commutative diagram induced by the localization sequence
    \begin{equation*}
  \begin{tikzcd}
\ldots\arrow{r} &\bigoplus_{D \subset Z}\CH^*_\et(X \times D) \arrow{r} \arrow{d}{\bigoplus f\otimes D}&\CH^*_\et(X \times Z) \arrow{r} \arrow{d}{f\otimes Z} & \CH^*_\et(X_K) \arrow{d}{(f_K)_*} \arrow{r}  & \ldots\\
\ldots\arrow{r} &\bigoplus_{D \subset Z}\CH^*_\et(Y \times D) \arrow{r} &\CH^*_\et(Y \times Z) \arrow{r} & \CH^*_\et(Y_K) \arrow{r} & \ldots
  \end{tikzcd}
\end{equation*}
where $K=k(Z)$. By assumption, we have that the map $\left(f_\Omega\right)_*:\CH^*_\et(X_\Omega)\to \CH^*_\et(Y_\Omega)$ is surjective, then we have that in the level of torsion $\CH_\et^*(X)_\text{tors} \to \CH_\et^*(X_\Omega)_\text{tors}$ is an isomorphism, so the map $\CH_\et^*(X_K)_\text{tors} \to \CH^*_\et(X_{\bar{K}})_\text{tors} \xrightarrow{\sim} \CH_\et^*(X_\Omega)_\text{tors}$ is surjective because it factors through the previous map. This gives us that the map $\left(f_K\right)_*:\CH^*_\et(X_K)_{\text{tors}}\to \CH^*_\et(Y_K)_{\text{tors}}$ is surjective, and then also the induced map $\left(f_K\right)_*:\CH^*_\et(X_K)\to \CH^*_\et(Y_K)$.

By induction hypothesis $f\otimes D$ it has the same image as $q \otimes D$, for all $D\subset Z$. In the same way, we have a similar commutative diagram involving $\bigoplus_{D\subset Y} q \otimes D$, $q \otimes Z$ and $q \otimes K$.

    \begin{equation*}
  \begin{tikzcd}
\ldots\arrow{r} &\bigoplus_{D \subset Z}\CH^*_\et(Y \times D) \arrow{r} \arrow{d}{\bigoplus q\otimes D}&\CH^*_\et(Y \times Z) \arrow{r} \arrow{d}{q\otimes Z} & \CH^*_\et(Y_K) \arrow{d}{(q_K)_*} \arrow{r}  & \ldots\\
\ldots\arrow{r} &\bigoplus_{D \subset Z}\CH^*_\et(Y \times D) \arrow{r} &\CH^*_\et(Y \times Z) \arrow{r} & \CH^*_\et(Y_K) \arrow{r} & \ldots
  \end{tikzcd}
\end{equation*}

Finally, as we have $\text{im}(q \otimes Z) =\text{im}(f \otimes Z)$, then $(f \times \text{id}_Z)_*:\CH^*_\et(M\otimes h_\et(Z)) \to \CH^*_\et(N\otimes h_\et(Z))$ is a surjective map for all $Z$ smooth projective variety, therefore by Lemma \ref{kim6.8} we have that $f$ is surjective.

\end{proof}

Finally, we can get the extension to the integral étale case of \cite[Lemma 1.1]{Huy} using the following lemma:

\begin{lemma}\label{lemsurj}
    Let $f:M\to N$ be a morphism of motives over $k$ such that for a universal domain $\Omega$, the induced morphism $\left(f_\Omega\right)_*:\CH^i_\et(M_\Omega) \to \CH^i_\et(N_\Omega)$ is an isomorphism for all $i\geq 0$. Then $f_\Omega$ is an isomorphism in the category $\ch_\et(\Omega)$.
\end{lemma}

\begin{proof}
    Let $\Omega$ be an universal domain of $k$.  By assumption, we have an isomorphism 
    $\left(f_\Omega\right)_*:\CH^i_\et(M_\Omega) \xrightarrow{\simeq} \CH^i_\et(N_\Omega)$, so by Lemma \ref{surj}, there exists a morphism $g: N_\Omega \to M_\Omega$ such that $f_{\Omega}\circ g = \text{id}_{N_\Omega}$. Therefore, we have a suboject of $M_\Omega$, denoted by $T$, and an isomorphism $f_\Omega:M_\Omega \to N_\Omega \oplus T$. Since $\CH^i_\et(M_\Omega) \xrightarrow{\simeq} \CH^i_\et(N_\Omega)$, then we obtain that $\CH^i_\et(T)\simeq 0$ for all $i \geq 0$, so invoking Lemma \ref{lemmM0}, we obtain that $T=0$, so we obtain that $f_\Omega : M_\Omega \to N_\Omega$ is an isomorphism.
\end{proof}

\begin{theorem}[Improved version of Manin's principle]\label{immanin}
Let $f:M\to N$ be a morphism in the category $\ch_\et(k)$. Then $f$ is a isomorphism of motives in $\ch_\et(k)$ if and only if for $\Omega$ an universal domain over $k$, the induced map $(f_\Omega)_*:\CH^*_\et(M_\Omega)\to\CH^*_\et(N_\Omega)$ given by the base change $f_\Omega: M_\Omega \to N_\Omega$, is bijective.
\end{theorem}

\begin{proof}
 Assume that $f: M\to N$ is a isomorphism and $K/k$ a field extension of $k$, then it is clear that $f_K:M_K \to N_K$ is an isomorphism in $\ch_\et(K)$. Now let us assume that $(f_\Omega)_*:\CH^*_\et(M_\Omega)\to\CH^*_\et(N_\Omega)$ is an isomorphism. By Lemma \ref{lemsurj}, then the map $f_\Omega: M_\Omega \to N_\Omega$ is bijective, then if we invoke \cite[Théorème 3.9]{Ayo} and the full-faithfulness of the functor $\ch_\et(k)^{op}\hookrightarrow \text{DM}_\et(k,\Z)$, we obtain that the associated functor $i^*: \ch_\et(k) \to \ch_\et(\Omega)$ is conservative, since $i^*(f)=f_\Omega$, we conclude that $f$ is an isomorphism in $\ch_\et(k)$.
\end{proof}

\section{Decomposition of abelian varieties}

The idea of this section is the construction of the decomposition for a group scheme $G$ over a noetherian base of finite dimension $S$ in the category $\text{DM}_\et(S)$. 
Let us recall that for such $S$, by \cite[Proposition 5.4.12]{CD16}, the family of functors
\begin{align*}
    \rho_\Q: \text{DM}_\et(S,\Z) &\to \text{DM}_\et(S,\Q) \\
    \rho_{\Z/\ell}:\text{DM}_\et(S,\Z) &\to \text{DM}_\et(S,\Z/\ell), \hspace{2mm}\ell \text{ prime number invertible in }S
\end{align*}
is conservative. Since $\text{DM}_\et(S,\Q) \simeq \text{DM}(S,\Q)$ and by rigidity theorem $ \text{DM}_\et(S,\Z/\ell) \simeq D(S_\et,\Z/\ell)$, the last being the derived category of étale sheaves with coefficients over $\Z/\ell$. 
Also there exists a duality functor $D_k$ in $D_{c}^b(k_\et,\Z/\ell^n)$, the so called Poincaré-Verdier duality, such that $D_k \circ D_k \simeq \text{Id}$ in a canonical way, according to \cite{sga5} and \cite{gab}. These results are more general i.e. working over a quasi-excellent scheme $S$. Notice that, for a morphism $p:X \to k$ and  constructible elements $M \in D_{c}^b(k_\et,\Z/\ell^n)$ and $N \in D_{c}^b(X_\et,\Z/\ell^n)$, we have the following isomorphism 
\begin{align*}
    D_k\left(p_!(N)\right)& \simeq p_*\left(D_X(N)\right), \quad D_X\left( p^! (M)\right) \simeq p^*\left( D_k (M)\right), \\
    D_k(p^*(M)) &\simeq p^!(D_X(M)), \quad p_!(D_X(N)) \simeq D_X(p_*(N)),
\end{align*}
therefore we obtain the element $D_k(p_! p^! (\Z/{\ell^n})) \simeq Rp_* p^* D(\Z/\ell^n) \in D_c(k_\et,\Z/\ell^r).$ The same kind of duality result exists for constructible elements $M \in \text{DM}_\et(k,\Q)$ and $N \in \text{DM}_\et(X,\Q)$  according to \cite[Corollary 4.4.24]{CD19}.

For a field $k$ and a prime number $\ell$ different from the characteristic of $k$, let us consider the following commutative diagram 
\[  
    \begin{tikzcd}
  \text{DM}_{\et}(k)\arrow{r}{\rho_{\Q}} \arrow{d}{\rho_\ell} & \text{DM}_{\et}(k,\Q) \arrow{d}{} \\
   \text{DM}_{\et}(k)^{\wedge \ell} \arrow{r}{} &  \text{DM}_{\et}(k)^{\wedge \ell}\otimes \Q
  \end{tikzcd}
\]
by the rigidity result the category $ \text{DM}_{\et}(k)^{\wedge \ell}$ is equivalent to $D(k_\et,\Z_\ell)$. Recall that for a motive $M_\et(X) \in \text{DM}_\et(k)$, the $\ell$-adic realization of $M_\et(X)$ is given by $\rho_\ell(M_\et(X))=p_! p^! (\Z_{\ell})$ where $p:X \to \spc(k)$ is the structural morphism, and this also works for finite coefficients $\rho_{\ell^n}(M_\et(X))=p_! p^! (\Z/\ell^n)$. 

Let $S$ be a noetherian finite dimensional scheme and let $G/S$ a smooth commutative group scheme of finite type over $S$. We start with the definition of the 1-motive associated to $G/S$, for that we define the étale sheaf induced by $G/S$:
\begin{defi}\label{1motive}
    Let $\underline{G/S}$ be the étale sheaf of abelian groups on $\text{Sm}_S$ defined by $G$:
    \begin{align*}
        \underline{G/S}(U)=\ho_{\text{Sm}_S}(U,G)
    \end{align*}
    for $U \in \text{Sm}_S$. 
    The 1-motive associated to $\underline{G/S}$ and is defined as $$M_1(G/S):= \Sigma^\infty M^\eff_1(G/S)  \in \textbf{DA}^\et(S,\Z),$$ where $M^\eff_1(G/S)$ is the effective étale motive in  $\textbf{DA}^\et_\eff(S,\Z)$ induced by $\underline{G/S}$.
\end{defi}

According to \cite[Theorem 3.7]{AHP} in the motivic category $\text{DM}_\et(S,\Q)$  we have a decomposition of the relative motive $M_S(G)$ in the following way
\begin{align*}
    M_S(G) \xrightarrow{ \simeq } \left( \bigoplus_{n\geq 0}^{\text{kd}(G/S)} \text{Sym}^n M_1(G/S) \right)\otimes M(\pi_0(G/S)),
\end{align*}
where $M_1(G/S)$ is the 1-motive induced by the étale sheaf represented by $\underline{G/S} \otimes \Q$ and $\text{kd}(G/S):=\max\left\{2g_s+r_s \ | \ s \in S\right\}$ is the Kimura dimension ($g_s$ is the abelian rank of $G_s$ and $r_s$ is the torus rank).

\begin{defi}\label{order}
The order of $\pi_0(G/S)$, denoted  by $o(\pi_0(G/S))$ is defined as the least common multiple of the order of all the elements of the groups $\pi_0(G_{\bar{s}}/\bar{s})$, with $\bar{s}$ geometric point of $S$.
\end{defi}

The aim of this subsection is to see if we can lift this isomorphism to integral coefficients in $\text{DM}_\et(S,\Z)$.  If we want to construct such a morphism in $\text{DM}_\et(S,\Z)$, we have to define the integral analogue of the \textit{symmetric algebra}. For this we consider the homotopy fixed points of a group action, where the group is finite. 

\subsubsection{\texorpdfstring{$\mathfrak{S}_n$}{Lg}-actions on \texorpdfstring{$\text{DM}_\et(S,\Z)$}{Lg}}

In this subsection we will present some aspects about the action of a finite group $G$ in the category $\text{DM}_\et(S,\Z)$ of integral étale motives. In this context, an $\infty$-category will be an $(\infty,1)$-category in the sense of Lurie \cite{Lur}. An $\infty$-functor between two $\infty$-categories $\mathcal{C}$ and $\mathcal{D}$ is simply a map $F: \mathcal{C} \to \mathcal{D}$ of simplicial sets.

Consider the group of permutations of n elements $\mathfrak{S}_n$ and let $B\mathfrak{S}_n$ be the category of a single object and morphism the elements of the groups $\mathfrak{S}_n$. Define the \textbf{homotopy fixed points}  and \textbf{homotopy orbits} of $\mathfrak{S}_n$ of a motive $M_\et^S(X)$ as follows: we know that $\text{DM}_\et(S,\Z)$ carries a structure of an $\infty$-category. Let $\text{DM}_\et(S,\Z)^{B\mathfrak{S}_n}$ be the category of étale motives with a 
$\mathfrak{S}_n-$action, i.e. $\infty$-functors $B\mathfrak{S}_n \to \text{DM}_\et(S,\Z)$. We obtain adjunctions 
 \begin{align*}
 ( \ )^{\text{triv}}: \text{DM}_\et(S,\Z)&\leftrightarrows \text{DM}_\et(S,\Z)^{B\mathfrak{S}_n}: ( \ )^{h\mathfrak{S}_n}:=\holim_{B\mathfrak{S}_n},\\    \hocolim_{B\mathfrak{S}_n}=:( \ )_{h\mathfrak{S}_n}: \text{DM}_\et(S,\Z)^{B\mathfrak{S}_n}&\leftrightarrows \text{DM}_\et(S,\Z): ( \ )^{\text{triv}}
 \end{align*}
where $ ( \ )^{\text{triv}}$ represents the trivial action. Let 
$\text{DM}_\et(S,\Z)^\otimes$ be the underlying symmetric monoidal category of $\text{DM}_\et(S,\Z)$. We  can give an explicit description of $(-)^{h\mathfrak{S}_n}$ for some motives by using the monoidal structure of $\text{DM}_\et(S,\Z)^\otimes$. Notice that for $X \in \text{Sm}_k$ we have an action of  $\mathfrak{S}_n$  given by 
 \begin{align*}
\sigma_*:X^n &\to X^n \\
     (x_1,\ldots,x_n) &\mapsto (x_{\sigma(1)},\ldots,x_{\sigma(n)})
 \end{align*}
where $\sigma \in \mathfrak{S}_n$ and $X^n:=\overbrace{X \times \ldots \times X}^{\text{n-times}}$. For such $X$ and $n$, consider the functor 
\begin{align*}
    F^n_X: B\mathfrak{S}_n& \to \text{Sm}_S \\
    * &\mapsto X^n \\
    (*\xrightarrow{\sigma} *) &\mapsto  (X^n\xrightarrow{\sigma_*} X^n) 
\end{align*}
We can consider the motive $M^S_\et(X)^{\otimes n}$ as an $\infty$-functor from the category $\text{Sm}_S$ to $\text{DM}_\et(S,\Z)^\otimes$. Therefore we obtain the homotopy fixed points of $M_\et^S(X)^{\otimes n}$ as 
\begin{align}\label{descrip}
    \left(M_\et^S(X)^{\otimes n}\right)^{h\mathfrak{S}_n}\simeq \holim_{B\mathfrak{S}_n} M_\et^S \circ F^n_X(-).
\end{align}

If we $\Q$-linearize the homotopy fixed points, then we have the following result relating them with the usual fixed points of a group action: 

\begin{lemma}\label{rational}
Let $M_S(X^n)\simeq M_S(X)^{\otimes n} \in \text{DM}_\et(S,\Q)$, then $$\left(M_S(X)^{\otimes n}\right)^{h\mathfrak{S}_n}\simeq \left(M_S(X)^{\otimes n}\right)^{\mathfrak{S}_n}$$ and equals the image of the projector $\frac{1}{n!}\sum_{\sigma \in \mathfrak{S}_n} \sigma_*M_S(X)^{\otimes n}$.
\end{lemma}

\begin{proof}
    This holds in greater generality, see \cite[3.3.21]{CD19}. If 
    $\mathcal{V}$ is a $\Q$-linear stable model category and $G$ is a finite group that acts on an object $E  \in \mathcal{V}$ we define $$E^{hG}:= \holim_{BG} E,$$ and define $E^G \in Ho(\mathcal{V})$ as the image of 
    \begin{align*}
        p(x)= \frac{1}{|G|}\sum_{g \in G}g.x.
    \end{align*}
Then the morphism $E^G \xrightarrow{\sim} E^{hG}$ induced by the inclusion $E^G \to E$ is an isomorphism in $Ho(\mathcal{V})$.
\end{proof}

\begin{remark}\label{remrational}
    \begin{enumerate}
        \item The same argument works in a category $\text{DM}_\et(S,\Lambda)$ if $n$ is invertible in the ring $\Lambda$. A very important remark is that the proof of the previous lemma relies in the commutative structure of the $\Q$-linear vector space. If we consider the homotopy fixed points using an anti-commutative structure (and $n$ is invertible), then we obtain that $E^{h\mathfrak{S}_n}$ equals the image of the projector  $\frac{1}{n!}\sum_{\sigma \in \mathfrak{S}_n} \text{sgn}(\sigma)\sigma_*E$.
        
        \item In the same way, we define the homotopy orbits of $\mathfrak{S}_n$ as the co-invariants of $M^S_\et(X)^{\otimes n}$, i.e. in the following way $(M^S_\et(X)^{\otimes n})_{h \mathfrak{S}_n} = \hocolim_{B\mathfrak{S}_n} M^S_\et(X)^{\otimes n}$. By the definition of homotopy colimit, we have a map $M^S_\et(X)^{\otimes n} \to (M^S_\et(X)^{\otimes n})_{h \mathfrak{S}_n}$.  
        
        \item Let $\text{DM}_h(S,\hat{\Z}_\ell)$ be localizing subcategory of $\text{DM}_h(S,\Z)$ generated by the objects of the form $M/\ell = \Z/\ell \otimes^R M$. According to \cite[7.2.10]{CD16} we have an adjunction $\hat{\rho}^*_\ell: \text{DM}_h(S,\Z) \leftrightarrows \text{DM}_h(S,\hat{\Z}_\ell):\hat{\rho}_{\ell *}$, where $\hat{\rho}^*_\ell$ is called the $\ell$-adic realization functor, which by  \cite[Theorem 7.2.11]{CD16} it is compatible with the six functors formalism of Grothendieck, and preserves colimits. Let $D(S_\et,\Z_\ell)$ be the derived category of $\ell$-adic sheaves as in \cite{Eke}. Consider the equivalence of categories  given in \cite[Proposition 7.2.21]{CD16}, then $\text{DM}_h(S,\hat{\Z}_\ell) \simeq  D(S_\et,\Z_\ell)$, so we define the realization functor
        $\rho_\ell: \text{DM}_h(S,\Z)\to  D(S_\et,\Z_\ell)$. This functor again is compatible with the six functors formalism of Grothendieck, and preserves colimits (as it is the composition of a left adjoint with an ), thus  we have that
\begin{align*}
    \rho_\ell\left((M^S_\et(X)^{\otimes n})_{h \mathfrak{S}_n}\right) \simeq \left((\rho_\ell M^S_\et(X))^{\otimes n}\right)_{h \mathfrak{S}_n}.
\end{align*}
    \end{enumerate}
\end{remark}

For the sake of completeness, we will present a reminder about the theory of 1-motives, with such goal in mind, we present some of the main results of \cite{org}.  Consider a commutative group scheme $G$ over a perfect field $k$. According to \cite[Lemme 3.1.1]{org},  the sheaf $\underline{G}$ is an étale presheaf which admits transfers. Notice that as a presheaf $\underline{G}$ is homotopy invariant i.e. $\underline{G}(U)\xrightarrow{\sim} \underline{G}(U\times \mathbb{A}^1_k)$ is an isomorphism for any $k$-smooth variety $U$, for a proof see \cite[Lemme 3.3.1]{org}. As we have a morphism of a complex of sheaves $C_*^{\text{Sus}}\Z^{tr}(G)\to C_*^{\text{Sus}}(\underline{G})$ and a quasi-isomorphism $\underline{G} \to C_*^{\text{Sus}}(\underline{G})$, finally we obtain a morphism in $\text{DM}_\et(k,\Z)$ between the motives $M_\et(G) \xrightarrow{\alpha_A} M_1(G)$.

If we want to work over a noetherian base $S$, and simplify the condition about étale sheaves with transfers, we work with the category $\textbf{DA}^\et(S,\Z)$. Consider a commutative group scheme $G$ over a noetherian base $S$ and let $\underline{G/S}$ the associated abelian sheaf and $M_1(G/S)$ the 1-motive described in Definition \ref{1motive}. As we have a morphism of a complex of pre-sheaves $a_{G/S}:\Z\ho_S(\cdot, G)\to \underline{G/S}$, then after sheafification we obtain a morphism $\alpha_{G/S}^\eff: M^{S,\eff}_\et(G) \to M_1^\eff(G/S)  \in \textbf{DA}^\et_\eff(S,\Z)$. Finally we obtain a morphism in $\textbf{DA}^\et(S,\Z)$ between the motives $\alpha_{G/S}=\Sigma^\infty \alpha_{G/S}^\eff: M_\et^S(G) \xrightarrow{\alpha_A} M_1(G/S)$ in $\textbf{DA}^\et(S,\Z)$.

As the functor $M^S_\et$ is monoidal and commutative, we obtain an isomorphism $M^S_\et(G\times G)\simeq M^S_\et(G)\otimes M^S_\et(G)$. For the general case we denote $M^S_\et(G)^{\otimes n}:=M^S_\et(\overbrace{G\times \ldots \times G}^{\text{n-times}})$. For a fixed $n$ and using the $n$-diagonal morphism $\delta^n_{G/S}:G \to G\times \ldots \times G$, we obtain an induced morphism of motives
 \begin{align*}
     M^S_\et(G) \xrightarrow{(\delta^n_{G/S})_*} M^S_\et(G)^{\otimes n}.
 \end{align*}
Together with the map $\alpha_G$ we construct a map
$$\phi_n: M^S_\et(G) \xrightarrow{(\delta^n_{G/S})_*} M^S_\et(G)^{\otimes n} \xrightarrow{\alpha^{\otimes n}_{G/S}} M_1(G/S)^{\otimes n}.$$ Notice that $M^S_\et(G)^{\otimes n}$ admits an action of the permutation group $\mathfrak{S}_n$, and this action leaves invariant the diagonal map $\delta^n_{G/S}$. Therefore we can apply the functor of homotopy fixed points $h\mathfrak{S}_n$. With this, we have a commutative diagram
 \[
  \begin{tikzcd}
M^S_\et(G) \arrow{r}{(\delta^n_{G/S})_*} \arrow[swap]{rd}{(\delta^n_{G/S})_*}& M^S_\et(G)^{\otimes n} \arrow{r}{\alpha^{\otimes n}_{G/S}} & M_1(G/S)^{\otimes n}  \\
&  \left(M^S_\et(G)^{\otimes n}\right)^{h\mathfrak{S}_n} \arrow{r}{ \alpha^{\otimes n}_{G/S}} \arrow{u} &  \left(M_1(G/S)^{\otimes n}\right)^{h\mathfrak{S}_n}. \arrow{u}
  \end{tikzcd}
\]
We denote the composite map $\alpha^{\otimes n}_{G/S} \circ(\delta^n_{G/S})_*: M^S_\et(G) \to \left(M_1(G/S)^{\otimes n}\right)^{h\mathfrak{S}_n}$ as $\phi^n_{G/S}$ .
\begin{defi}\label{mor}
    Let $G$ be a smooth commutative group scheme over a noetherian scheme $S$. Then we define the following:
        \item \begin{align*}
            \phi_{G/S}:= \bigoplus_{i\geq 0}\phi^i_{G/S}: M^S_\et(G) \to \bigoplus_{i\geq 0} \left(M_1(G/S)^{\otimes i}\right)^{h\mathfrak{S}_i}  
        \end{align*}
 in the category $\text{DM}_\et(S,\Z)$. We define the weak symmetric algebra of $M_1(G)$ as
\begin{align*}
    \text{wSym}(M_1(G)):=\bigoplus_{i\geq 0}\left(M_1(G)^{\otimes i}\right)^{h\mathfrak{S}_i}.
\end{align*}        
\end{defi}


Let us give some information about the realization of the morphism $\phi_{G/S}$ presented in Definition \ref{mor}, in the category $D(S_\et,\Z/\ell^n)$ for a prime number $\ell$ invertible in $S$ and $n \in \N$. We have a realization $\rho_{\Z/\ell^n}:\text{DM}_h(S,\Z) \to D(S_\et,\Z/\ell^n)$ and denote $M_1(G/S,\ell^n):= \rho_{\Z/\ell^n}(M_1(G/S))$, which is a complex in degree $-1$. We use $\mathcal{H}_1(G/S,\Z/\ell^n)$ for the homology of $M_1(G/S,\ell^n)$ is degree 1 and $\mathcal{H}^1(G/S,Z/\ell^n)$ for the cohomology of the complex $M_1(G/S,\ell^n)$ in degree $-1$. If the base is $S=\spc(k)$ with $\bar{k}=k$ then $D(k_\et,\Z/\ell) \simeq (\mathbb{F}_\ell-v.s.)^\Z$, where $(\mathbb{F}_\ell-v.s.)^\Z$ is the category of $\Z$-graded $\mathbb{F}_\ell$-vector spaces, $\mathcal{H}_1(G/S,\Z/\ell^n)$  is a finite dimensional $\mathbb{F}_\ell$-vector space.


\begin{lemma}\label{realization}
    Let $\ell$ be a prime number invertible in $S$, $n \in \N$ and consider the realization functor $\rho_{\Z/\ell^n}: \text{DM}_\et(S,\Z) \to D(S_\et,\Z/\ell^n)$. Then
    \begin{enumerate}
        \item if $\phi: \text{DM}_\et(S,\Z) \to \text{DM}_\et(S,\Z) $ is an additive functor and $\bar{\phi}:D(S_\et,\Z/\ell^n) \to D(S_\et,\Z/\ell^n)$ its associated counterpart with finite coefficients, then the functor $\rho_{\Z/\ell^n}$ commutes with $\phi$, in the sense that $\bar{\phi} = \rho_{\Z/\ell^n} \circ \phi$.
        \item If $S=\spc(k)$ for some field $k$, then $\rho_{\Z/\ell^n}(M_1(G))=\mathcal{H}_1(G,\Z/\ell^n)\simeq G[\ell^n][1]$\footnote{Here the first square bracket is  associated to the $\ell^n$-torsion of $G$ and the second is associated to the translation functor.}.
        \item There exists $N>>0$ such that $(M_1(G/S)^{\otimes m})^{h \mathfrak{S}_m}=0$ for all $m>N$.
    \end{enumerate}
\end{lemma}

\begin{proof}
 1. Recall that the functor is defined as $\rho_{\Z/\ell^n}(M)=\Z/\ell^n \otimes^L M = \text{coker}(M \xrightarrow{ \cdot \ell^n}M)$, therefore we have a canonical isomorphism $\rho_{\Z/\ell^n}(M) \simeq \text{Cone}( \ell^n \cdot \text{id}_M)$.   Let $\phi$ be an additive functor. Then in the  commutative diagram
 \[
  \begin{tikzcd}
\phi(M) \arrow{r}{\phi(\ell^n\cdot\text{id}_M)} \arrow[d,equal] & \phi(M) \arrow{r} \arrow[d,equal] & \bar{\phi}\left(\rho_{\Z/\ell^n}(M)\right) \arrow{d}  \xrightarrow{+1} \\
\phi(M) \arrow{r}{\ell^n\cdot\text{id}_{\phi(M)}} & \phi(M) \arrow{r} & \rho_{\Z/\ell^n}(\phi(M)) 
\xrightarrow{+1} 
  \end{tikzcd}
\]
the right vertical arrow is an isomorphism as well.

 2.  Let us consider the 1-motive $M_1(G) = \underline{G}$ which is concentrated in degree 1. Recall that the $\ell$-adic realization of $M_1(G)$, integral or rational, is given by the Tate module $T_\ell(G)=\varprojlim_n G[\ell^n]$, thus $\rho_{\Z/\ell^n}(M_1(G))\simeq G[\ell^n][1]$ by the transition maps.

 3. Using Lemma \ref{rational}, we see that the weak symmetric algebra of $M_1(G/S)$ with rational coefficients coincides with the symmetric algebra of $M_1(G/S)$. In particular $\text{Sym}^n(M_1(G/S))=0$ in $\text{DM}_\et(k,\Q)$ if $n>\text{kd}(G/S)$ by \cite[Proposition 4.1]{AHP}. The only argument that remains to be given is for the torsion part. For this, consider a prime number $\ell$ invertible in $S$. Notice that $\text{wSym}(\mathcal{H}_1(G/S,\Z/\ell))$ is anticommutative by the cup-product, see \cite[Proposition 7.4.10]{Fu15}, therefore according to the first point in Remark \ref{remrational}, if $n$ and $\ell$ are coprime, we have that $(\mathcal{H}_1(G/S,\Z/\ell)^{\otimes n})^{h \mathfrak{S}_n}\simeq \bigwedge^n \mathcal{H}_1(G/S,\Z/\ell)$, and in particular vanishes if $n>\text{kd}(G/S)$. If $n$ and $\ell$ are not coprime, then we proceed as follows: we can reduce to the case where $S=\spc(k)$ for an algebraically closed field $k$. Then by the point 2, we have that $M_1(G/S,\ell^n)$  is a complex in degree 1 whose first homology group $\mathcal{H}_1(G/S,\Z/\ell)$ is a finite dimensional vector space over $\mathbb{F}_\ell$. Let $r_\ell$ be the dimension of $\mathcal{H}_1(G/S,\Z/\ell)$ and let $\{e_1,\ldots,e_{r_\ell}\}$ be a base, then if we consider $m>r_\ell$, then there will always be at least one $e_i$ repeated in $e_{i_1}\otimes \ldots \otimes e_{i_m}$ as a base of $\mathcal{H}_1(G/S,\Z/\ell)^{\otimes m}$. By the alternating action of $\mathfrak{S}_m$ in $\mathcal{H}_1(G/S,\Z/\ell)^{\otimes m}$ and the description in this particular case of the homotopy fixed points given in (\ref{descrip}), we conclude that  $(\mathcal{H}_1(G/S,\Z/\ell)^{\otimes m})^{h\mathfrak{S}_m}=0$ for all $m>r_\ell$. 
 
 Since the family of functors associated to the change of coefficient is conservative, then one concludes that $N=\max_\ell\{\text{kd}(G),r_\ell\}$.
\end{proof}

\begin{remark}
\begin{enumerate}
    \item Consider $S=\spc(k)$ with $k$ an algebraically closed field, then for a commutative algebraic group $G/k$. By an argument given in \cite[Proposition 4.1]{BSz} involving reduction to the assumption that $G$ is semi-abelian, we may assume that $G$ is the extension of an abelian variety $A$ by a torus $T$, we have a short exact sequence $1\to T[\ell^n] \to G[\ell^n] \to A[\ell^n] \to 1$ obtaining that $G[\ell^n] \simeq (\Z/\ell^n)^{2g+r}$, where $g$ is the dimension of $A$ and $r$ is the rank of $T$.
    \item Under the same assumptions for $S$, thank to the second point of Lemma \ref{realization} we get that $\rho_{\Z/\ell}(M_1(G))\in D(k_\et,\Z/\ell) \simeq (\mathbb{F}_\ell-v.s.)^\Z$, where $(\mathbb{F}_\ell-v.s.)^\Z$ is the category of $\Z$-graded $\mathbb{F}_\ell$-vector spaces, is a finite dimensional $\mathbb{F}_\ell$-vector space. Since the dimension of the vector space depends only on $G$ and not on $\ell$, we can say that the $N$ described in point 3 of \ref{realization} corresponds to the Kimura dimension $\text{kd}(G)$.
\end{enumerate}
    
\end{remark}

\begin{lemma}\label{lemmacoeffin}
    Let $G$ be a smooth group scheme over a field $k=\bar{k}$ and let $\ell$ be a prime number different from $\text{char}(k)$. Then we have an isomorphism in $D(k_\et,\Z/\ell) \simeq (\mathbb{F}_\ell-e.v.)^{\Z}$ given by
    \begin{align*}
        \rho_{\Z/\ell}(M_\et(G))=M_\et(G)/\ell \xrightarrow{\sim} \bigoplus_{i=0}^{\text{kd}(G)}\left(\bigwedge^i \mathcal{H}_1(G,\Z/\ell)\right)[i]    \end{align*}
\end{lemma}
\begin{proof}
First, we have that $M_1(G)$ is a geometric motive and is $\Z$-additive, therefore we have $M_1(G\times H)\simeq M_1(G)\oplus M_1(H)$. Let us recall that  the weak symmetric algebra of $M_1(G/S)$ is defined as
\begin{align*}
    \text{wSym}(M_1(G)):=\bigoplus_{i=0}^{\text{kd}(G)}\left(M_1(G)^{\otimes i}\right)^{h\mathfrak{S}_i}.
\end{align*}
 By point 1 of Lemma \ref{realization}, $\rho_{\Z/\ell}$ commutes with any additive functor, so we get
\begin{align*}
    \rho_{\Z/\ell}\left(  \bigoplus_{i=0}^{\text{kd}(G)} \left(M_1(G)^{\otimes i}\right)^{h\mathfrak{S}_i}\right)\simeq  \bigoplus_{i=0}^{\text{kd}(G)}\rho_{\Z/\ell}\left(  \left(M_1(G)^{\otimes i}\right)^{h\mathfrak{S}_i}\right).
\end{align*}
Notice that by definition  $\rho_{\Z/\ell}(M)=M\otimes^L \Z/\ell$. Since $-\otimes^L \Z/\ell$ commutes with colimits and is monoidal, see \cite[Definition 5.6]{Ayo}, we obtain an isomorphism 
\begin{align*}
    \rho_{\Z/\ell}\left(\left(M_1(G)^{\otimes i}\right)^{h\mathfrak{S}_i}\right)& \simeq \left(\rho_{\Z/\ell}\left(M_1(G)^{\otimes i}\right)\right)^{h\mathfrak{S}_i} \\
    &\simeq \left( \mathcal{H}_1(G,\Z/\ell)^{\otimes i}\right)^{h\mathfrak{S}_i}.
\end{align*}

So in terms of realization we obtain
$\rho_{\Z/\ell}(\text{wSym}(M_1(G)))\simeq \text{wSym}( \mathcal{H}_1(G,\Z/\ell))$. Notice that if $\ell > \text{kd}(G)$, then we get immediately (using the proof in point 3 of \ref{realization}) that $ \text{wSym}( \mathcal{H}_1(G,\Z/\ell)) \simeq \bigoplus_{i=0}^{\text{kd}(G)}\left(\bigwedge^i \mathcal{H}_1(G,\Z/\ell)\right)[i]$.

In the following part we will prove that for a prime $\ell \neq \text{char}(k)$ and $G$, $H$ two smooth commutative algebraic groups we have an isomorphism $\text{wSym}( \mathcal{H}_1(G \times H,\Z/\ell))\simeq \text{wSym}( \mathcal{H}_1(G,\Z/\ell))\otimes\text{wSym}(  \mathcal{H}_1(H,\Z/\ell))$, which will allow us to conclude when $\ell\leq \text{kd}(G)$. By definition 
\begin{align*}
    \text{wSym}(M_1(G\times H))&\simeq \text{wSym}(M_1(G) \oplus M_1(H))\\
    &=\bigoplus_{i=0}^{\text{kd}(G\times H)}\left(\left(M_1(G) \oplus M_1(H)\right)^{\otimes i}\right)^{h\mathfrak{S}_i}
\end{align*}
Since the homotopy fixed points of a motive $M$ are defined as a homotopy limit, they commute with finite sums. For simplicity we write $M:= M_1(G)$ and $N:=M_1(H)$, thus we have $\left(M^{\otimes n}\oplus N^{\otimes n}\right)^{h\mathfrak{S}_n}\simeq \left(M^{\otimes n} \right)^{h\mathfrak{S}_n}\oplus \left(N^{\otimes n} \right)^{h\mathfrak{S}_n}$. Moreover we have a canonical morphism
\begin{align*}
    \holim_{B\mathfrak{S}_n}(M\oplus N)^{\otimes n} \to (M\oplus N)^{\otimes n}\simeq \bigoplus_{i=0}^n \bigoplus^{\binom{n}{i}} M^{\otimes i} \otimes N^{\otimes n-i},
\end{align*}
where the last isomorphism is obtained by the distributive and commutative properties. Since $\rho_{\Z/\ell}$ commutes with additive functors and limits and is monoidal, we obtain $\rho_{\Z/\ell}\left( \left((M\oplus N)^{\otimes n}\right)^{h\mathfrak{S}_n}\right)\simeq \left((M/\ell\oplus N/\ell)^{\otimes n}\right)^{h\mathfrak{S}_n}$, passing to the realization $\rho_{\Z/\ell}$ and due to the anticommutative structure given by the cup-product, if $n$ and $\ell$ are coprimes, then  
\begin{align*}
\left((M/\ell\oplus N/\ell)^{\otimes n}\right)^{h\mathfrak{S}_n} \simeq \bigwedge^n (M/\ell\oplus N/\ell).
\end{align*}

Notice that the functor $B\mathfrak{S}_i \times B\mathfrak{S}_j \to B(\mathfrak{S}_i \times \mathfrak{S}_j)$ is an equivalence of categories, then 
\begin{align*}
    ((M/\ell)^{\otimes i}\otimes (N/\ell)^{\otimes n-i})^{h(\mathfrak{S}_i \times \mathfrak{S}_{n-i})}& \simeq \holim_{B\mathfrak{S}_i \times B\mathfrak{S}_{n-i}}(M/\ell)^{\otimes i}\otimes (N/\ell)^{\otimes n-i}\\
    &\simeq ((M/\ell)^{\otimes i})^{h\mathfrak{S}_i}\otimes ((N/\ell)^{\otimes n-i})^{h\mathfrak{S}_{n-i}}.
\end{align*}

Due to the anticommutativity of the weak algebra, we have an isomorphism of $\mathbb{F}_\ell$-vector spaces
 \[
  \begin{tikzcd}
\holim_{B\mathfrak{S}_n} \overbrace{M^{\otimes i}\otimes N^{n-i} \oplus \ldots \oplus M^{\otimes i}\otimes N^{n-i}}^{\binom{n}{i}\text{-times}} \arrow{r}\arrow{d}{\simeq} & \bigoplus^{\binom{n}{i}}M^{\otimes i}\otimes N^{\otimes n-i} \\
\holim_{B(\mathfrak{S}_i \times \mathfrak{S}_{n-i})}M^{\otimes i}\otimes N^{n-i} \arrow{ur} &
  \end{tikzcd}
\]


Therefore, we have an isomorphism of graded (anticommutative) algebras $\text{wSym}( \mathcal{H}_1(G \times H,\Z/\ell))\simeq \text{wSym}( \mathcal{H}_1(G,\Z/\ell))\otimes\text{wSym}(  \mathcal{H}_1(H,\Z/\ell)).$ 

 So the final argument is that the addition map $m:G \times G \to G$ induces a morphism of graded algebras over the field $\mathbb{F}_\ell$
\begin{align*}
m^*:\text{wSym}(M_1(G)/\ell) &\to  \text{wSym}(M_1(G)/\ell)\otimes\text{wSym}( M_1(G)/\ell)\\
x &\mapsto x \otimes 1 + 1\otimes x + \sum x_i \otimes y_i.
\end{align*}
By a lemma about the fundamental structure of such algebras, \cite[Lemma 15.2]{Mil84}, we have that $\text{wSym}(M_1(G)/\ell) \simeq \bigoplus_{i=0}^{\text{kd}(G)}\left(\bigwedge^i \mathcal{H}_1(G,\Z/\ell)\right)[i] $.
\end{proof}

 By using the properties of conservative functors associated to change of coefficients described in \cite[Proposition 5.4.12]{CD16}, which for the sake of completeness, we recall such proposition:
 \begin{prop}[{\cite[Proposition 5.4.12]{CD16}}]
 Let $\mathcal{P}$ be the set of prime integers and $S$ be a noetherian scheme of finite dimension. If $R$ is a flat ring over $\Z$, then the family of change of coefficients functors:
 \begin{align*}
     &\rho_\Q: \text{DM}_h(S,R) \to \text{DM}_h(S, R\otimes \Q) \\
     &\rho_{\Z/p}: \text{DM}_h(S,R) \to \text{DM}_h(S, R/p), \quad p \in \mathcal{P}
 \end{align*}
 is conservative.
 \end{prop}
 
With this proposition, we get an improvement of the results obtained in \cite{AEWH}, getting the following theorem:

\begin{theorem}\label{teoprin}
 Let $k$ be an algebraically closed field and $G/k$ a connected commutative group scheme. Then the morphism    
 \begin{align*}
      \phi_G:M_\et(G) \to \bigoplus_{i=0}^{\text{kd}(G)} \left(M_1(G)^{\otimes i}\right)^{h\mathfrak{S}_i}
 \end{align*}
 is an isomorphism in $\text{DM}_\et(k,\Z)$.
\end{theorem}

\begin{proof}
We split the proof into two steps: first we start by looking at the functor $\rho_\Q$. Applying Lemma \ref{rational}, we obtain that the induced morphism by $\rho_\Q(\phi_G)$ is the morphism $\varphi_G$ given in \cite[Definition 3.1]{AHP} and \cite[Theorem 3.3]{AHP}, with $S=k$, which is shown to be an isomorphism in $\text{DM}_\et(k,\Q)$. The reason behind this is the following: $\rho_\Q(M_\et(G)) = M(G)$ and by Lemma \ref{rational} we have $\rho_\Q\left(\bigoplus_{i=0}^{\text{kd}(G)} \left(M_1(G)^{\otimes i}\right)^{h\mathfrak{S}_i}\right) \simeq \bigoplus_{i=0}^{\text{kd}(G)} \text{Sym}^n(M_1(G))$ and finally, by construction of the morphisms $\phi_G$ and $\varphi_G$ of Definition \ref{mor} and \cite[Definition 3.1]{AHP}, and the uniqueness part of \cite[Theorem 2.8]{AHP} we get that $\rho_\Q(\phi_G)=\varphi_G$.

For the second step, we fix a prime number $\ell \neq \text{char}(k)$. Let us consider the functor $\rho_{\Z/\ell}$ and let us compute the elements of
\begin{align*}
    \rho_{\Z/\ell}(\phi_G): \rho_{\Z/\ell}\left( M_\et(G)\right) \to \rho_{\Z/\ell}\left(  \bigoplus_{i=0}^{\text{kd}(G)} \left(M_1(G)^{\otimes i}\right)^{h\mathfrak{S}_i}\right).
\end{align*}

Here we are assuming that the realization functor is covariant (sending the elements to homological instead of cohomological objects).  By \cite[Theorem 4.1]{BSz} we have that 
\begin{align*}
   H^*_\et(G,\Z/\ell) \simeq \bigoplus_{i=0}^{\text{kd}(G)}\left(\bigwedge^i \mathcal{H}^1(G,\Z/\ell)\right)[i], 
\end{align*}
where $\ell$ is a prime number not equal to $\text{char}(k)$. Consider the duality operator $D_k$ on $D(k_\et,\Z/\ell)$, let $f:G \to k$ be the structure morphism and $d$ the dimension of $G$. By definition we have 
\begin{align*}
 D_k\left(\rho_{\Z/\ell}\left( M_\et(G)\right) \right)&\simeq D_k f_! f^! (\Z/\ell)_k \\
 &\simeq f_* D_G (\Z/\ell)_G \\
 &\simeq (M_\et(G)_\ell)(d)[2d].
\end{align*}
Here the first isomorphism $\rho_{\Z/\ell}\left( M_\et(G)\right)\simeq f_! f^! (\Z/\ell)_k$ is because the  realization functor commutes with the six functors formalism, see \cite[A.1.16]{CD16}, while the second are third are given by \cite[Exposé XVII]{gab}

As we stated in Lemma \ref{lemmacoeffin}, one has that $\rho_{\Z/\ell}(M_\et(G))\simeq \bigoplus^{\text{kd}(G)}_{i=0}\left(\bigwedge^i \mathcal{H}_1(G,\Z/\ell) \right)[i]$ for all $\ell \neq \text{char}(k)$, whose dual is isomorphic to $\bigoplus_{i=0}^{\text{kd}(G)}\left(\bigwedge^i \mathcal{H}^1(G,\Z/\ell)\right)[i]$,  thus by conservative properties given in \cite[Proposition 5.4.12]{CD16}, we conclude that 
\begin{align*}
    M_\et(G) \xrightarrow{\simeq} \bigoplus_{i=0}^{\text{kd}(G)} \left(M_1(G)^{\otimes i}\right)^{h\mathfrak{S}_i} \in \text{DM}_\et(k,\Z).
\end{align*}


\end{proof}

\begin{theorem}\label{relat}
    Let $S$ be a good enough scheme in the sense of Definition \ref{gooden}, and let $G$ be a connected commutative scheme over $S$. Then the morphism $\phi_G$ given in Definition \ref{mor} is an isomorphism.
\end{theorem}

\begin{proof}
    Consider a morphism of good enough schemes $f:T \to S$, we have that $f^* M_1(G/S) \simeq M_1(G_T/T) \in \text{DM}_\et(T,\Z)$. As we have following, we will split the proof in two: first we have it for $\text{DM}_\et(T,\Q)$ and then for $\text{DM}_\et(T,\Z/\ell)$ for all prime integer $\ell$ invertible in $T$.  According to \cite[Proposition 2.7]{AHP}, one has that $f^* M_1(G/S)_\Q \simeq M_1(G_T/T)_\Q \in \text{DM}_\et(T,\Q)$. On the other hand, having shown that $\rho_{\Z/\ell^n}(M_1(G/S))\simeq \underline{G/S}[\ell^n][-1]$, and that by the universal property of fibre product we have $f^*\underline{G/S}\simeq \underline{G_T/T}$. Invoking \cite[Théorème 6.6(A)]{Ayo} one gets for a quasi-projective morphism $f^* \circ \rho_{\Z/\ell^r} \simeq \rho_{\Z/\ell^r}\circ f^*$ , thus we get the following isomorphism 
    \begin{align*}
        \rho_{\Z/\ell^n}(f^*M_1(G/S))&\simeq f^*\left( \rho_{\Z/\ell^n}(M_1(G/S))\right) \\
        &\simeq f^*\left(\underline{G/S}[\ell^n][1]\right) \\
        &\simeq f^*\left(\underline{G/S}\right)[\ell^n][1]\\
        &\simeq \underline{G_T/T}[\ell^n][1] =  \rho_{\Z/\ell^n}(M_1(G_T/T)).
    \end{align*}
    As $\ell$ is any prime number, then we conclude that $f^* M_1(G/S) \simeq M_1(G_T/T) \in \text{DM}_\et(T,\Z)$. In particular we obtain that $\phi_G$ is natural over the base. By the previous fact, we get that the morphism $\phi_{G_T}:=f^*(\phi_G)$ acts as follows:
    \begin{align*}
        M^S_\et(G) \xrightarrow{\phi_G}  \bigoplus_{i=0}^{\text{kd}(G/S)} \left(M_1(G/S)^{\otimes i}\right)^{h\mathfrak{S}_i}  \leadsto  M^T_\et(G_T) \xrightarrow{\phi_{G_T}}  \bigoplus_{i=0}^{\text{kd}(G_T/T)} \left(M_1(G_T/T)^{\otimes i}\right)^{h\mathfrak{S}_i}. 
    \end{align*}
    
    In this way, for any geometric point $i_{\bar{s}}:\bar{s}\to S$ we have an isomorphism of motives $i^*_{\bar{s}}M_1(G/S) \simeq M_1(G_{\bar{s}}/\bar{s})$, and then by the previous remark, $i_{\bar{s}}^*(\phi_G)=\phi_{G_{\bar{s}}}$ for any geometric point $\bar{s}$ of $S$, the map
    \begin{align*}
        M^{\bar{s}}_\et(G_{\bar{s}}) \xrightarrow{\phi_{G_{\bar{s}}}}  \bigoplus_{i=0}^{\text{kd}(G_{\bar{s}}/{\bar{s}})} \left(M_1(G_{\bar{s}}/{\bar{s}})^{\otimes i}\right)^{h\mathfrak{S}_i}
    \end{align*}
    turns out to be an isomorphism by Theorem \ref{teoprin}. By Lemma \ref{consCD}, the family of functors $i_{\bar{s}}^*$ is conservative, therefore $\phi_{G/S}$ is an isomorphism. 
    \end{proof}

\begin{remark}
 The direct factor $h_n(G/S)=\phi_{G/S}^{-1}\left(\left(M_1(G/S)^{\otimes n}\right)^{h\mathfrak{S}_n}\right)$ of $M^S_\et(G)$ is characterized as follows: for $m \in \Z$ that is equals to 1 modulo $o(\pi_0(G/S))$ (see Definition \ref{order}), the map $M_\et([m])$ operates on $h_n(G/S)$ as $m^n \cdot \text{id}$. This is a consequence of \cite[Lemma 2.6(1)]{AHP}.
\end{remark}

\begin{theorem}
Suppose that $S$ is a good enough scheme. Then for every bi-degree $(m,n)\in \Z^2$ the relative étale cohomology groups of $G$ in degrees $(m,n)$ with integral coefficients decomposes as
\begin{align*}
    H_{M,\et}^m(G/S,\Q(n)) \simeq \bigoplus_{j=0}^{\text{kd}(G/S)} \ H_{M,\et}^{m,j}(G/S,\Q(n)),
\end{align*}
where $$H_{M,\et}^{m,j}(G/S,\Q(n))=\left\{Z \in H_{M,\et}^m(G/S,\Q(n)) \ | \ [m]^*Z = m^j Z, \ \forall m \equiv 1 (\text{mod} \ o(\pi_0(G/S)))\right\}.$$
\end{theorem}

\begin{proof}
This follows from the decomposition $M_\et^S(G)= \bigoplus_{n=0}^\text{kd(G/S)}h_n(G/S)$ thanks to Theorem \ref{relat} and the definition of étale motivic cohomology.
\end{proof}

Let $A$ be an abelian variety over a field $k$ whose cohomological dimension if finite, the question which arises naturally is if the isomorphism $M_\et(A)\to \bigoplus_{i=0}^{\text{kd}(A)} \left(M_1(A)^{\otimes i}\right)^{h\mathfrak{S}_i}$ comes from an morphism in $\ch_\et(k)$.

\begin{prop}
    Let $A$ be an abelian variety of dimension $g$ over a field $k$ with $\text{cd}(k)<\infty$, then the following are equivalent:
    \begin{enumerate}
        \item the isomorphism $M_\et(A)\to\bigoplus_{i=0}^{2g} \left(M_1(A)^{\otimes i}\right)^{h\mathfrak{S}_i} $ is a morphism in the category $\ch_\et(k)$.
        \item The exist an element $h \in \ch_\et(k)$ such that $h \simeq M_1(A) \in \ch_\et(k)$.
    \end{enumerate}
\end{prop}

\begin{proof}
Let $A$ be an abelian variety over $k$ of dimension $g$, and consider the map $$M_\et(A) \to \bigoplus_{i=0}^{2g} \left(M_1(A)^{\otimes i}\right)^{h\mathfrak{S}_i} \in \text{DM}_\et(k,\Z).$$ 

Let us remark that $h_\et(A)\simeq M_\et(A)$ by the full embedding of $\ch_\et(k)^{op}\hookrightarrow  \text{DM}_\et(k,\Z)$. If we assume (1), then (2) follows immediately since $M_1(A)$ is a direct factor of the finite sum $\bigoplus_{i=0}^{2g} \left(M_1(A)^{\otimes i}\right)^{h\mathfrak{S}_i} \in \ch_\et(k)$.

If we assume (2), exists an element $h\in \ch_\et(k)$ such that $h\simeq M_1(A)$, then $h^{\otimes i} \in \ch_\et(k)$ for all $i\geq 0$. We will assume that $h=(X,p)$ for some $X \in \text{SmProj}_k$ and $p \in \text{Corr}_\et^0(X,X)$. In other words, $p \in \text{End}_{\ch_\et(k)}(h(X)) \simeq \text{End}_{\text{DM}_\et(k,\Z)}(M(X))$, therefore $p^{\otimes n} \in \text{End}_{\ch_\et(k)}(h(X)^{\otimes n})$, thus we define $(p^{\otimes n})^{h\mathfrak{S}_n}$ as the image of $p^{\otimes n}$ in $\text{End}_{\text{DM}_\et(k,\Z)}((M(X)^{\otimes n})^{h\mathfrak{S}_n})$, we then define the motive
\begin{align*}
    \left(h^{\otimes n}\right)^{h\mathfrak{S}_n}:= (\overbrace{X \times \ldots \times X}^{n\text{-times}},(p^{\otimes n})^{h\mathfrak{S}_n}) \in \ch_\et(k)
\end{align*}
Then we see that morphism $M_\et(A) \to \bigoplus_{i=0}^{2g} \left(M_1(A)^{\otimes i}\right)^{h\mathfrak{S}_i}$ is in $\ch_\et(k)$.
\end{proof}

Changing the coefficients in the proof of Theorem \ref{teoprin}, we obtain that $h_\et(A)$ admits an integral Chow-Künneth decomposition in $\ch_\et(k)$ if the 1-motive $M_1(A)$ belongs to the category $\ch_\et(k)$. The previous statement can be seen as well as a consequence of the pseudo-abelianity of the category $\ch_\et(k)$.

In order to give an example of an integral étale motive with Chow-Künneth decomposition, we should recall some results coming from the classical theory of abelian varieties. We have the following results:
\begin{lemma}\label{variab}
    Let $C$ be a smooth projective curve over a finite cohomological dimension field $k$ with a $k$-rational point, and let $J(C)$ be the Jacobian of $C$, then:
    \begin{enumerate}
        \item the motive $M_\et(C)$ can be decomposed as
        $M_\et(C)\simeq \mathbf{1} \oplus h^1_\et(C) \oplus \mathbf{1}(1)[2].$
        \item  If $C'$ is another smooth projective curve over $k$, then 
        $$\ho_{\ch_\et(k)}\left(h^1_\et(C),h^1_\et(C')\right)\simeq \ho_{AV}\left(J(C),J(C')\right)[1/p].$$
        \item The motives $h_\et^1(C)$ and $M_1(J(C))$ are isomorphic.
    \end{enumerate}
\end{lemma}
\begin{proof}
1. It is a classic result, for instance see \cite[Theorem 2.7.2]{MNP} and the fully-faithful functor of 1 motives to the étale cohomology with integral coefficients.

2. This is a consequence of the isomorphism $$\ho_{\ch(k)_\Z}\left(h^1(C),h^1(C')\right)[1/p]\simeq \ho_{\ch_\et(k)}\left(h^1_\et(C),h^1_\et(C')\right)$$ 
and \cite[Theorem 2.7.2.(b)]{MNP}.

3. The argument is the same as in \cite[Lemma 4.3.2]{AEWH}. Consider the 1-motive $M_\et(C)$ which is cohomologically concentrated in degrees 0 and -1 as is given in \cite[Theorem 3.4.2]{V00}. The cohomology in degree 0 is $\underline{\text{Pic}_{C/k}}[1/p]$ while in degree 0 is equal to $\underline{\mathbb{G}_m}[1/p]$. Since $\Z_C(1)_\et \sim \underline{\mathbb{G}_m}[1/p]_\et[-1]$ we obtain that $\mathbf{1}\oplus \mathbf{1}(1)[2]\simeq \Z[1/p] \oplus \underline{\mathbb{G}_m}[1/p][1]$. The remaining object is given by the kernel of the map $\underline{\text{Pic}_{C/k}}[1/p] \to \Z[1/p]$, which is isomorphic to $M_1(J(C)).$
\end{proof}

\begin{theorem}\label{prodJac}
Let $k$ be a field of finite cohomological dimension and consider $C_i/k$ a  projective smooth curves with a $k$-rational point for $i \in \{1,\ldots,n\}$. Then the variety $J(C_1)\times \ldots \times J(C_n)$ admits an integral Chow-Künneth decomposition. 
\end{theorem}

\begin{proof}
 By point 3. of Lemma \ref{variab}, one has an isomorphism $h_\et^1(C_i)\simeq M_1(J(C_i))$. Since $M_1$ is an additive functor, we obtain $M_1(J(C_1)\times  \ldots \times J(C_n))\simeq \bigoplus_{i=1}^n M_1(J(C_i))$, thus $M_1(J(C_1)\times  \ldots \times J(C_n)) \simeq\bigoplus_{i=1}^n h_\et^1(C_i)$ in $\ch_\et(k)$. Therefore $M_1(J(C_1)\times  \ldots \times J(C_n))$ is isomorphic to the motive $h=\left(\coprod_i C_i,\sum_{i=1}^n p^1_\et(C_i)\right)\in \ch_\et(k)$
\end{proof}
    

Recall that thanks to \cite{Kun} we have Chow-Künneth decomposition with rational coefficients for abelian varieties and that by \cite[Proposition 4.3.3]{AEWH} the $h_1(A)$ part is isomorphic to $M_1(A)_\Q$ in $\ch(k)_\Q$. Given the part $h_1(A)$ of the motive $h(A)$ and its associated projector $p_1(A)$, then we can characterize the existence of an integral étale Chow-Künneth decomposition in the following way:

\begin{theorem}\label{teoProjAb}
 Let a field $k=\bar{k}$ and consider $A$ an abelian variety of dimension $g$ over $k$. Then the following statements are equivalents
 \begin{enumerate}
     \item $h_\et(A)$ admits an integral étale Chow-Künneth decomposition.
     \item the projector associated to $h_1(A) \in \ch(k)_\Q$ can be lifted to a projector in $\CH^g_\et(A\times A)$.
 \end{enumerate}    
\end{theorem}

\begin{proof}
1. $(\Longrightarrow)$ 2. is immediate. If we assume 2. then there exists an element $p \in \CH^g_\et(A \times A)$ such that $h_1(A)=(A,p_\Q) \in \ch(k)_\Q$ where $p_\Q$ is the image of $p$ in $\CH^g_\et(A\times A)_\Q$. Then the realization of the motive $h=(A,p)$ coincides with $h_1(A)$ if we change to rational coefficients. If $\ell \neq \text{char}(k)$, then $ H_{1}(A_\et,\Z_\ell) = T_\ell(A)$ is $\Z_\ell$-torsion free since $A[\ell^n]\simeq (\Z/\ell^n)^{2g}$, so we have an injection  $H_{1}(A_\et,\Z_\ell) \hookrightarrow T_\ell(A) \otimes \Q_\ell = H_{1}(A_\et,\Q_\ell)$, therefore $p_\Q$ acts as the identity on $H_{1}(A_\et,\Z_\ell)$. 

Consider the realization $p_{\ell^n} \in \CH^g_\et(A\times A,\Z/\ell^ n)$, where the last group is isomorphic to $H^{2g}_\et(A\times A,\Z/\ell^n)$. As $\Q_\ell$ is a flat $\Z_\ell$-module and $\varprojlim_{n \in \N}$ is a right exact functor we have that $p_{\ell^n}$ acts as the identity over $H_1(A_\et,\Z/\ell^n)$.
\end{proof}

\begin{theorem}
 Let $k=\bar{k}$ be a field and let $A$ be an abelian variety. Then there exists a Chow-Künneth decomposition of $A$.
\end{theorem}

\begin{proof}
Consider an abelian variety $A/k$, then we have that
\begin{align*}
    \text{End}_{\ch_\et(k)}(h_\et(A)) \simeq  \text{End}_{\text{DM}_\et(k)}(M_\et(A)). 
\end{align*}
Since $M_1(A)$ is a direct factor of $M_\et(A)$, then it defines an endomorphism $p$ of $M_\et(A)$, as the endomorphism of the motive $h_\et(A)$ is defined as $\CH^g_\et(A\times A)$ where $g=\text{dim}(A)$. Since $p \in \CH^g_\et(A\times A)$ such that $p^2=p$, thus we define the motive $h_1(A):=(A,p,0)$. The functor $\ch_\et(k)^{op}\hookrightarrow \text{DM}_\et(k)$, sends $h_1(A)\mapsto M_1(A)$, therefore, as $M_1(A)\in \ch_\et(k)$ we conclude that $h_\et(A)$ admits a Chow-Künneth decomposition.
\end{proof}

\begin{theorem}
 Let $X$ be a smooth projective variety of dimension $d$ over an algebraically closed field $k$. If $\text{Pic}^0(X)$ is a principally polarized variety, then there exists a decomposition of the motive $h_\et(X)$ as
 \begin{align*}
     h_\et(X)=h_\et^0(X)\oplus h^1_\et(X) \oplus h^+_\et(X) \oplus  h^{2d-1}_\et(X) \oplus h_\et^{2d}(X) 
 \end{align*}
 \end{theorem}

\begin{proof}
An abelian variety $A$ admits a principal polarization if and only $\widehat{A}$ admits one. So if $A = \text{Pic}^0(X)$ then the Picard variety is principally polarized if and only if $\text{Alb}_X(k)$ admits one. Let $k$ be an algebraically closed field and let $A$ be a principally polarized abelian variety $\lambda: A \xrightarrow{\sim} \widehat{A}$ induced by a symmetric ample line bundle. We have an injection
\begin{align*}
    \ho_{AV}(A,\widehat{A}) \hookrightarrow \ho_{\Z_\ell}(T_\ell(A),T_\ell(\widehat{A}))
\end{align*}
Therefore $\lambda$ induce an isomorphism of Tate modules. Since we have the following isomorphisms
\begin{align*}
    H^{2g-1}_\et(A,\Z_\ell)\simeq T_\ell(A) \xrightarrow{\sim} T_\ell(\widehat{A}) \simeq H^{2g-1}_\et(\widehat{A},\Z_\ell)
\end{align*}
considering the isomorphism $H^1_\et(A,\Z_\ell)\simeq H^{2g-1}_\et(\widehat{A},\Z_\ell)$. Since $\lambda$ is induced by a cycle (thanks to the integral étale Fourier transform). Now take an hyperplane $H$ in $X$ and intersect it with itself $g-1$ times, then it induces a Lefschetz operator
\begin{align*}
    L^{g-1}_A: H^1_\et(X,\Z_\ell) \to H^{2g-1}_\et(X,\Z_\ell)
\end{align*}
which turns out to be an injection. We will see that there exists and étale cycle in $\CH^g_\et(A\times A)$ whose multiple by an integer equals the Lefschetz operator. We recall that there exists isomorphisms
\begin{align*}
    \ho^0_{\text{SmProj}_k}(X,\text{Pic}^0(X)) &\simeq \ho_{AV}(\text{Alb}_X(k),\text{Pic}^0(X)) \\
   & \simeq \CH^1_\et(X\times X)/\CH^1_{\equiv}(X\times X)
\end{align*}
where  $\CH^1_{\equiv}(X\times X) = \text{pr}^*_1(\CH^1_\et(X))\oplus\text{pr}^*_2(\CH^1_\et(X))$ and $ \ho^0_{\text{SmProj}_k}$ stands for the pointed morphisms of smooth projective varieties over $k$. Thus the polarization $\lambda:\text{Alb}_X(k) \to \text{Pic}^0(X)$ is induced by a divisor in $X \times X$. Now consider the abelian variety $A=\text{Pic}^0(X)$, then there is a morphism $\lambda^{-1}:\text{Pic}^0(X) \to \text{Alb}_X(k)$, thanks to the existence of a Fourier transform with integral coefficients which is motivic. The cycle $c_1(\mathcal{P}_{\text{Pic}^0(X)})^{2g-1}/(2g-1)! \in \CH^{2g-1}_\et(\text{Pic}^0(X)\times \text{Alb}_X(k))$ where $g=\text{dim}(\text{Pic}^0(X))$ induce an isomorphism $H^1(X,\Z/\ell^n) \to H^{2d-1}(X,\Z/\ell^n)$, since we have isomorphisms $f:\text{Pic}^0(X)[\ell^n]\simeq H^1(X,\Z/\ell^n)$ and $g:\text{Alb}_X(k)[\ell^r]\simeq H^{2d-1}(X,\Z/\ell^n)$, see \cite[Chap. III, Cor. 4.18]{Mil80}, thus we associate the cycle 
\begin{align*}
\lambda^{-1}:= g \circ \frac{c_1(\mathcal{P}_{\text{Pic}^0(X)})^{2g-1}}{(2g-1)!} \circ f^{-1} \in \CH^{d}_\et(X\times X).
\end{align*}

By arguments given in \cite[Lemma 6.2.3]{MNP}, one has that $L_X^{g-1}$ defines an isogeny $\alpha : \text{Pic}^0(X)\to \text{Alb}_X(k)$ and another one $\beta: \text{Alb}_X(k) \to \text{Pic}^0(X)$ such that $\alpha \circ \beta = m \cdot \text{id}_{\text{Alb}_X(k)}$ and $\beta \circ \alpha = m \cdot \text{id}_{\text{Pic}^0(X)}$ for some $m \in \N$. If we take the isogeny $\lambda: \text{Alb}_X(k) \to \text{Pic}^0(X)$ and $\lambda^{-1}:\text{Pic}^0(X) \to \text{Alb}_X(k)$, with this, we have that $\lambda^{-1}\circ \lambda = \text{id}_{\text{Alb}_X(k)}$ and $\lambda\circ \lambda^{-1} = \text{id}_{\text{Pic}^0(X)}$, since  $\alpha \circ \beta $ is induced by an algebraic cycle and also $\lambda^{-1}\circ \lambda$ (but in this specific case it is induced by an étale cycle). Therefore $m$ is invertible in $\CH_\et^g(X\times X)$, thus we obtain the existence of the projector $p_1^\et(X)$ and $p_{2d-1}^\et(X)$

\end{proof}

We will end up by saying a few words about the Chow-Künneth decomposition of conic bundles following \cite{NS}. The main reference for Prym varieties and conic bundles that we considered is \cite{bea}.  Let us consider a field of characteristic different from 2. Let $\pi: \tilde{C}\to C$ be a double étale cover between smooth projective curve of genus $g(C)=g$ and $g(\tilde{C})=2g-1$, the norm morphism $\text{Nm}_\pi:J(\tilde{C})\to J(C)$ which takes an invertible sheaf $\mathcal{O}_{\tilde{C}}(\tilde{D})$ to the invertible sheaf  $\mathcal{O}_{C}(\pi_*\tilde{D})$. The kernel of the map has two connected component. We define the \textit{Prym variety} of the cover $\pi$ denoted by $P_{\tilde{C}}$ to be the connected component of $0_{J(\tilde{C})}$. The Prym variety is principally polarized induced by the canonical principal polarization on $J(\tilde{C})$. 

We assume that $k$ is an algebraically closed field. Consider a threefold $f:X \to S$, we say that $X$ is a conic bundle if the fibers of $f$ are conics. There exists a curve $C$ called the discriminant of $f$. This morphism is flat and if $\omega_X$ is the canonical sheaf on $X$, then $\xi= f_* \omega^{-1}_X$ is a locally free sheaf of rank 3 and $X$ is defined in $\mathbb{P}_S(\xi)$. Because of this we have isomorphisms
\begin{align*}
    H^0(X)\simeq H^0(\mathbb{P}_S(\xi)), \ H^1(X)\simeq H^1(\mathbb{P}_S(\xi)), \ H^2(X)\simeq H^2(\mathbb{P}_S(\xi))
\end{align*}
by the Lefschetz hyperplane theorem and also we can characterize the groups $H^i(X)$ for $i=4,5$ and 6. Here $H^i(X)=H^i_\et(X,\Z_\ell)$ or $H^i(X)=H^i_B(X,\Z)$ if the base field is $k=\C$. Concerning the middle cohomology group we have that for $\ell \neq \text{char}(k)$
\begin{align*}
    H^3_\et(X,\Z_\ell) &\simeq H^3_\et(S,\Z_\ell)\oplus H^1_\et(S,\Z_\ell)(-1) \oplus H^1_\et(P_{\widetilde{C}},\Z_\ell)(-1)\text{ or } \\
     H^3_B(X,\Z) &\simeq H^3_B(S,\Z)\oplus H^1_B(S,\Z)(-1) \oplus H^1_B(P_{\widetilde{C}},\Z)(-1) \ \text{ if }k=\C.
\end{align*}

For a double covering $\sigma:\tilde{C}\to C$ which has an associated involution $\iota: \tilde{C}\to \tilde{C}$ with the identification $C=\tilde{C}/(\iota)$, consider the integral 1-motive $M_1(P_{\tilde{C}}) \in \text{DM}_\et(k,\Z)$. Consider the isomorphisms $h_\et^1(C)\simeq M_1(J(C))$ and $h_\et^1(\tilde{C})\simeq M_1(J(\tilde{C}))$. Since we have an isomorphism
\begin{align*}
    P_{\tilde{C}}= \text{ker}(\text{Nm}_\pi)^0 = \text{im}(1-\iota)\hookrightarrow J(\tilde{C})
\end{align*}
and the isomorphism at the level of hom-set in different categories:
\begin{align*}
    \ho_{\ch_\et(k)}\left(h^1_\et(\tilde{C}),h^1_\et(\tilde{C})\right)\simeq \ho_{AV}\left(J(\tilde{C}),J(\tilde{C})\right)[1/p]
\end{align*}
and 
\begin{align*}
    \ho_{AV}\left(J(\tilde{C}),J(\tilde{C})\right)[1/p] \simeq \ho_{\text{DM}_\et(k,\Z)}\left(M_1(J(\tilde{C})),M_1(J(\tilde{C}))\right),
\end{align*}
it follows that the map $\text{id}-\iota$ defines a morphism of étale motives $\text{id}-\iota:h^1_\et(\widetilde{C})\to h^1_\et(\widetilde{C}) \in \text{DM}_\et(k,\Z)$ induced by the cycle induced by the graph. Thus we define the integral Prym motive as $\text{Pr}(\tilde{C}/C):=\text{im}(1-\iota) \in \text{DM}_\et(k,\Z)$. Let $M$ be the motive defined as
\begin{align*}
    M:=h_\et(S) \oplus h_\et(S)(-1) \oplus \text{Pr}(\tilde{C}/C)(-1).
\end{align*}

\begin{remark}
\begin{enumerate}
    \item We have to remark the next points: We can obtain a fully characterization of the étale Chow groups of $X/\C$ with $S$ a smooth surface. Since $J^2(X)\simeq \text{Alb}_S(S)\oplus \text{Pic}^0(S)\oplus P_{\widetilde{C}}$ by \cite[Theorem 3.5]{belt} and using the results of \cite{bea} and \cite{belt}, we obtain the following characterization
\begin{align*}
    \CH^0_\et(X)&\simeq \Z, \\ \CH_\et^1(X)&\simeq \CH^1(X), \\
    \CH_\et^2(X)&\simeq \CH^2_\et(S)\oplus \CH^1_\et(S)\oplus P_{\widetilde{C}},\\
     \CH^3_\et(X)&\simeq \CH^3(X).
\end{align*}
Since we have an isomorphism $\CH^*(X)\simeq \CH_\et^*(X)$, by \cite[Theorem 1.1]{RS} we can conclude that the classical integral Hodge conjecture holds for smooth conic bundles  $X \to S$.

\item Now consider $X \to \Pro_{\C}^2$ a quadric bundle of relative dimension $2m-1$. There exist linear sections $X \hookrightarrow \Pro(\xi) \to \Pro_{\C}^2$, thanks to this fact, we obtain the cohomology groups of $X$
\begin{align*}
    H^i_B(X,\Z) \simeq H^i_B(\Pro(\xi),\Z) \text{ if }i<2m+1.
\end{align*}
Since we have an isomorphism  of torsion groups $J^k(X)_{\text{tors}}\simeq \CH_\et^k(X)^{\text{tors}}_{\text{hom}}$, thus we can conclude that $H^i_B(X,\Z)=0$ for $i\neq 2m+1$ an odd number, then $J^k(X)=0$, so $\CH_\et^k(X)^{\text{tors}}_{\text{hom}}=0$ and with that we obtain 
\begin{align*}
    \CH^{j}_\et(X)&\simeq \CH^0(\Pro^2_\C)\oplus \CH^1_\et(\Pro^2_\C) \oplus \CH^2_\et(\Pro^2_\C), \text{ if }j\neq m+1\\
    \CH^{m+1}_\et(X)&\simeq \CH^0(\Pro^2_\C)\oplus \CH^1_\et(\Pro^2_\C) \oplus \CH^2_\et(\Pro^2_\C) \oplus \CH^1_\et(\widetilde{C}),
\end{align*}
thus again by \cite[Theorem 1.1]{RS} the integral Hodge conjecture for 1-cycles holds for a quadric bundle $X\to \Pro_\C^2$.
\end{enumerate}
\end{remark}

\begin{corollary}
Let $k=\bar{k}$ be a field and let $f:X \to S$ be a conic bundle. Then we have an isomorphism
$h_\et(X) \simeq h_\et(S) \oplus h_\et(S)(-1) \oplus \text{Pr}(\tilde{C}/C)(-1)$. For a smooth projective surface $S$ which admits an integral étale Chow-Künneth decomposition, consider the motive $M=\bigoplus_{i=0}^6 M^i$ where each $M^i$ is defined as
\begin{align*}
    M^i :=\begin{cases}
    h^3_\et(S)\oplus h^1_\et(S)(-1) \oplus \text{Pr}(\tilde{C}/C)(-1) \text{ if }i=3\\
    h^i_\et(S)\oplus h^{i-2}_\et(S)(-1)\text{ if }i\neq 3
    \end{cases}
\end{align*}
then $h_\et(X)\simeq \bigoplus_{i=0}^6 M^i$.
\end{corollary}

\begin{proof}
Let us consider the following facts: for $X$ we have that the intermediate Jacobian $J^2(X)$ is an abelian variety isomorphic to $\text{Alb}_S(k)\oplus \text{Pic}^0(S)\oplus \mathcal{P}_{\tilde{C}}$, see \cite[Theorem 3.5]{belt}. Notice that $h^3(X) \in \ch(k)_\Q$ is a finite dimensional motive isomorphic to the 1-motive given by $M_1(J^2(X))_\Q$. Since we have an isomorphism of cohomology groups
\begin{align*}
    H^3_\et(X,\Z_\ell) &\simeq H^3_\et(S,\Z_\ell)\oplus H^1_\et(S,\Z_\ell)(-1) \oplus H^1_\et(P_{\widetilde{C}},\Z_\ell)(-1) 
\end{align*}
which is induced by an algebraic cycle since there is an isogeny $\phi: J^2(X)\to \text{Alb}_S(k)\oplus \text{Pic}^0(S)\oplus \mathcal{P}_{\tilde{C}}$, we take the cycle associated cycle $\Gamma_\phi^t$, this induces an isomorphism of motives $M_1(J^2(X))_\Q \simeq M_1(\text{Alb}_S(k)\oplus \text{Pic}^0(S)\oplus \mathcal{P}_{\tilde{C}})_\Q$. This also can be seen from the fact that we have an isomorphism of finite dimensional homological motives
\begin{align*}
  M_1(J^2(X))_\Q \simeq M_1(\text{Alb}_S(k)\oplus \text{Pic}^0(S)\oplus \mathcal{P}_{\tilde{C}})_\Q \in \ch_{\text{hom}}(k)  
\end{align*}
then the isomorphism can be lifted to $M_1(J^2(X))_\Q \simeq M_1(\text{Alb}_S(k)\oplus \text{Pic}^0(S)\oplus \mathcal{P}_{\tilde{C}})_\Q \in \ch(k)$. As this cycle exists up to rational coefficients, mixing the result \cite[Theorem 1.3.b]{RS} with the arguments given in \cite[Proposition 3.1.7]{RoSo}, we can conclude similarly that this cycle lifts to an integral étale cycles in different algebraically closed fields, concluding that 
\begin{align*}
    h^i_\et(X) \simeq \begin{cases}
    h^3_\et(S)\oplus h^1_\et(S)(-1) \oplus \text{Pr}(\tilde{C}/C)(-1) \text{ if }i=3\\
    h^i_\et(S)\oplus h^{i-2}_\et(S)(-1)\text{ if }i\neq 3.
    \end{cases}
\end{align*}
\end{proof}

\printbibliography[title={Bibliography}]
\info

\end{document}